\documentclass[graybox]{svmult}

\usepackage{type1cm}         \usepackage{makeidx}         \usepackage{graphicx}        \usepackage{multicol}        \usepackage[bottom]{footmisc}

\usepackage{amsfonts, amsmath, amssymb}
\usepackage{xspace}          \usepackage{ytableau}        \usepackage{pgf}
\usepackage{subcaption}
\usepackage{multirow}

\usepackage[draft=true]{minted} \setminted{breaklines=true}
\definecolor{lbcolor}{rgb}{0.9,0.9,0.9}
\setminted{bgcolor=lbcolor}
\setminted{fontsize=\footnotesize}

\usepackage{mathtools}
\usepackage{booktabs}        \usepackage{tikz}            \usetikzlibrary{calc}        \usepackage{csquotes}

\usepackage[colorinlistoftodos, bordercolor=orange, backgroundcolor=orange!20, linecolor=orange, textsize=scriptsize]{todonotes}

\usepackage[noend]{algorithmic}

\usepackage{calc,ifthen,alltt,tabularx,capt-of}
\usepackage[linesnumbered,longend,ruled,vlined]{algorithm2e}

\usepackage[
   natbib,
   style=alphabetic,
   sorting=nty,
   maxbibnames=5, maxalphanames=5, isbn=false,url=false]
{biblatex}
\usepackage{lineno}
\usepackage[bookmarks=false]{hyperref}

\usepackage{longtable}
\usepackage{makecell}
\usepackage{wrapfig}

\usepackage[capitalise, noabbrev]{cleveref} \usepackage{enumitem}

\newcommand\OSCAR{\texttt{OSCAR}\xspace}

 \newcommand\ZZ{\mathbb{Z}} \newcommand\QQ{\mathbb{Q}} 
\newcommand\RR{\mathbb{R}} \newcommand\CC{\mathbb{C}}

\DeclareMathOperator{\conv}{conv}

\definecolor{promptColor}{rgb}{0.0,0.0,0.589}
\definecolor{brkpromptColor}{rgb}{0.589,0.0,0.0}
\definecolor{gapinputColor}{rgb}{0.589,0.0,0.0}
\definecolor{gapoutputColor}{rgb}{0.0,0.0,0.0}
\definecolor{darkgreen}{rgb}{0.05,0.6,0.1}
\definecolor{colrem}{rgb}{0,0.7,0}

\spnewtheorem{algorithm2}{Algorithm}{\bfseries}{\itshape}

\DeclareMathOperator{\Trop}{Trop}

\DeclareMathOperator{\mult}{mult}
\DeclareMathOperator{\trop}{trop}

\DeclareMathOperator{\val}{val}

\newcounter{truefigure}
\makeatletter
\renewcommand{\p@subfigure}{}
\makeatother

\newcommand{\PreserveBackslash}[1]{\let\temp=\\#1\let\\=\temp}
\newcolumntype{C}[1]{>{\PreserveBackslash\centering$}p{#1}<{$}}
\newcolumntype{R}[1]{>{\PreserveBackslash\raggedleft$}p{#1}<{$}}
\newcolumntype{L}[1]{>{\PreserveBackslash\raggedright$}p{#1}<{$}}

\newcommand{\TT}{\mathbb{T}}
\newcommand{\stsect}{\wedge}
\newcommand{\lin}{{\text{lin}}}
\newcommand{\nlin}{{\text{nlin}}}

\DeclareMathOperator{\codim}{codim}
\DeclareMathOperator{\initial}{in}
\DeclareMathOperator{\relint}{relint}
\DeclareMathOperator{\spanoperator}{span}

\spnewtheorem{convention}{Convention}{\bfseries}{\rmfamily}

\makeindex

\begin{document}
\title*{Generic root counts of tropically transverse systems}
\subtitle{An invitation to tropical geometry in \OSCAR}
\author{Isaac Holt and Yue Ren}
\institute{Isaac Holt \at Durham University, Department of Mathematics, Durham DH1 3LE, United Kingdom \email{isaac.a.holt@durham.ac.uk}
  \and Yue Ren \at Durham University, Department of Mathematics, Durham DH1 3LE, United Kingdom \email{yue.ren2@durham.ac.uk}}
\maketitle
\abstract{The main mathematical focus of this paper is a class of para\-metrised polynomial systems that we refer to as being tropically transverse.  We show how their generic number of solutions can be expressed as the mixed volume of a modified system.  We then provide an alternate proof of a recent result by Borovik et al, on the number of equilibria of coupled nonlinear oscillators using elementary tropical geometry.  The proof draws upon a wide range of concepts across tropical geometry, which we will use as an opportunity to give a first overview over various tropical features in \OSCAR.}

\section{Introduction}\label{sec:introduction}
Tropical geometry is often described as a piecewise linear analogue of algebraic geometry.  Its concepts arise naturally in many applications both inside and outside mathematics.  Due to its explicit nature, algorithms and computation have always played a central role in tropical geometry.  Prominent examples include disproving strong polynomiality of certain barrier interior points methods in linear optimisation by computing tropicalisations of central paths \cite{ABGJ18,AGV22}, proving finiteness of certain central configurations in the $N$-body problem by computing tropical prevarieties of their algebraic equations \cite{HamptonMoeckel06,HamptonJensen11}, or breakthroughs in the geometry of moduli spaces by computing tropicalisations of linear series \cite{FJP20}.

This paper addresses the problem of computing the generic number of solutions of a certain class of parametrized polynomial systems, which we call \emph{tropically transverse}.  These are a specialisation of the parametrised polynomial systems considered by Kaveh and Khovanskii in their work on Newton-Okounkov bodies \cite{KavehKhovanskii2012}, and examples include the polynomials describing the stationary motion of coupled nonlinear oscillators studied by Borovik, Breiding, Del Pino, Michalek and Zilberberg \cite{BBPMZ23}.  One question of particular interest about parametrised polynomial system is that of their generic root count.  It may be interesting in its own right, such as the Euclidean distance degree \cite{DHOST14} or the maximum likelihood degree \cite{CHKS06}, but it is also a crucial piece of information for solving polynomial systems via numerical algebraic geometry \cite{BBCHLS23}.
We show that the generic root count of tropically transverse systems can be expressed as the mixed volume of a modified system in \cref{thm:tropicallyTransverseSystem}, and provide an alternate proof of \cite[Theorem 5.1]{BBPMZ23} using elementary tropical geometry in \cref{thm:mainSystem}.

Before tackling tropically transverse systems, this paper serves as an introduction and invitation to the tropical package in \OSCAR \cite{OSCAR}.  While the package is still rapidly growing and evolving, it brings as part of the greater \OSCAR and \texttt{julia} ecosystem a plethora of powerful and sophisticated algorithms within arm's reach of the tropical user.

All algorithms pertaining to tropically transverse systems are openly available under
\begin{center}
  \url{https://github.com/isaacholt100/generic_root_count}
\end{center}

\section{Tropical functionalities in OSCAR}\label{sec:tropicalOverview}
In this section, we give a very brief rundown on some tropical features in \OSCAR whilst introducing some necessary concepts for the rest of the paper.  Our notation generally follows \cite{MS2015}, though we may need to deviate from it for practical reasons.

\subsection{Tropical semirings}
First, we begin with the most fundamental building block of tropical geometry that is the tropical semiring.

\begin{definition}\label{def:tropicalSemiring}
  While tropical semirings in \cite[Section 1.1]{MS2015} are extensions of the real numbers, tropical semirings in \OSCAR are either
  \begin{align*}
    \TT_{\max}&\coloneqq (\QQ\cup\{-\infty\},\oplus,\odot)\text{ with }a\oplus b\coloneqq \max(a,b), a\odot b\coloneqq a+b \text{, or}\\
    \TT_{\min}&\coloneqq (\QQ\cup\{+\infty\},\oplus,\odot)\text{ with }a\oplus b\coloneqq \min(a,b), a\odot b\coloneqq a+b \text{ (default).}
  \end{align*}
  This allows \OSCAR to avoid precision issues that arise when working with real numbers.  $\TT_{\min}$ and $\TT_{\max}$ are referred to as the \emph{min-plus} or \emph{max-plus semiring}, respectively.  Their elements are called \emph{tropical numbers}.
\end{definition}

\begin{example}[Tropical semiring and tropical numbers]\label{ex:tropicalSemiring}
  Like \cite{MS2015}, \OSCAR follows the min-convention.  This means that the tropical semiring is the min-plus semiring by default:

  \noindent
\begin{minted}{jlcon}
julia> T = tropical_semiring()
Min tropical semiring

julia> T(1)+T(2), T(1)*T(2), zero(T), one(T)
((1), (3), TODOINFINITY, (0))

julia> T = tropical_semiring(max)
Max tropical semiring

julia> T(1)+T(2), T(1)*T(2), zero(T), one(T)
((2), (3), -TODOINFINITY, (0))
\end{minted}

  \noindent
  Note:
  \begin{enumerate}[leftmargin=*]
  \item Tropical numbers are printed in parenthesis to distinguish them from rational numbers.
  \item Rational numbers $0,1\in\QQ$ are converted to tropical numbers $0,1\in\mathbb{T}$, not the neutral elements of tropical addition and multiplication. Similarly, $-1\in\QQ$ is converted to $-1\in\mathbb{T}$, the tropical multiplicative inverse of $1\in\mathbb{T}$ and not the tropical additive inverse  which does not exist.
  \item Tropical semirings are ordered \cite[Section 2.7]{ETC}. For two rational numbers $a<b$ we have
  \[ a<b \text{ in } \TT_{\min}\quad\text{and}\quad a>b \text{ in } \TT_{\max}. \]
  Consequently, $\mp\infty$ are the smallest elements of their semiring, and the conversion from rational to tropical numbers preserves the ordering for $\TT_{\min}$ and reverses the ordering for $\TT_{\max}$.  The ordering is relevant for initial forms and ideals in \cref{def:initialIdeal}.
  \end{enumerate}
\end{example}

\begin{example}[Tropical matrices and polynomials]
  In \OSCAR, tropical semirings are concrete subtypes of \texttt{field} in which subtraction is undefined.  Albeit deviating from math semantics, it allows us to access all generic features thereover, such as matrices and polynomials:

  \noindent
\begin{minted}{jlcon}
julia> T = tropical_semiring()
Min tropical semiring

julia> M = matrix(T,[0 1; 2 3])
[(0)   (1)]
[(2)   (3)]

julia> v = T.([-1,-1])
2-element Vector{TropicalSemiringElem{typeof(min)}}:
 (-1)
 (-1)

julia> M^2*v
2-element Vector{TropicalSemiringElem{typeof(min)}}:
 (-1)
 (1)

julia> R,(x,y) = T["x","y"];

julia> f = 1*x^2+2*y^2+0
(1)*x^2 + (2)*y^2 + (0)

julia> f^2+2*f
(2)*x^4 + (3)*x^2*y^2 + (1)*x^2 + (4)*y^4 + (2)*y^2 + (0)

julia> evaluate(f,T.([1,1]))
(0)
\end{minted}

  \noindent
  However, functions whose generic implementation requires subtraction will raise an error at runtime.  In tropical mathematics, these concepts usually have subtractionless counterparts, such as the \emph{tropical determinant} \cite[Equation 1.2.6]{MS:2015} \cite[Chapter 3.1]{Joswig21},
  \[ \det(X)\coloneqq \bigoplus_{\pi\in S_n}x_{1\pi(1)}\odot\dots\odot x_{n\pi(n)}\quad\text{for }X=(x_{ij})_{i,j=1,\dots,n}\in\TT^{n\times n}, \]
  and providing implementations of specialised tropical functions will make \OSCAR avoid the generic implementation when called with tropical input:
\begin{minted}{jlcon}
julia> T = tropical_semiring()
Min tropical semiring

julia> M = matrix(T,[0 1; 2 3])
[(0)   (1)]
[(2)   (3)]

julia> det(M)
(3)
\end{minted}
\end{example}

\subsection{Valued fields}
Next, we introduce some tropical notions over valued fields, which are central to many applications of tropical geometry, in particular in and around algebraic geometry.

\begin{convention}\label{con:valuedField}
  For the remainder of the paper, let $K$ denote a field with a valuation, i.e., a field with a map $\val\colon K^\ast\rightarrow\RR$ satisfying
  \begin{enumerate}
  \item $\val(a\cdot b)=\val(a)+\val(b)$,
  \item $\val(a+b)\geq\min(\val(a),\val(b))$.
  \end{enumerate}
  This valuation naturally induces a \emph{ring of integers} $R\coloneqq \{c\in K\mid \val(c)\geq 0\}$, a maximal ideal $\mathfrak m\coloneqq \{c\in R\mid \val(c)>0\}\unlhd R$ and a \emph{residue field} $\mathfrak K \coloneqq R/\mathfrak m$.

  We further assume that there is a fixed splitting homomorphism\linebreak $\psi\colon (\val(K^\ast),+)\rightarrow (K^\ast,\cdot)$ such that $\val(\psi(r))=r$.  Such splitting homomorphism exists for all algebraically closed fields \cite[Lemma 2.1.15]{MS2015}, but also many others.  For example:
  \begin{itemize}
  \item If $\val$ is trivial, i.e., $\val(c)=0$ for all $c\in K^\ast$, then $\{0\}\rightarrow K^\ast, 0\mapsto 1$ is a viable splitting homomorphism.
  \item If $K=\QQ_p$ is the field of $p$-adic numbers and $\val$ is the usual $p$-adic valuation, then $\ZZ\rightarrow\QQ_p^\ast, r\mapsto p^r$ is a viable splitting homomorphism.
  \item If $K=\QQ(t)$ is the field of rational functions in $t$ and $\val$ is the $t$-adic valuation, then $\ZZ\rightarrow\QQ(t)^\ast, r\mapsto t^r$ is a viable splitting homomorphism.
  \end{itemize}
  Following the notation of \cite{MS2015}, we abbreviate $t^r\coloneqq\psi(r)$.
\end{convention}

\begin{definition}\label{def:initialIdeal}
  To each polynomial $f\in K[x]\coloneqq K[x_1,\dots,x_n]$, say $f=\sum_{\alpha\in\ZZ_{\geq 0}^n}c_\alpha x^\alpha$ with $c_\alpha\in K$, we can assign a tropical polynomial $\trop(f)\in\TT[x]$ defined as follows:
  \begin{equation*}
    \trop(f)\coloneqq
    \begin{cases}
      \displaystyle\bigoplus_{c_\alpha\neq 0}\val(c_\alpha)\odot x^{\odot \alpha} &\text{if }\TT=\TT_{\min}, \\[7mm]
      \displaystyle\bigoplus_{c_\alpha\neq 0}-\val(c_\alpha)\odot x^{\odot \alpha} &\text{if }\TT=\TT_{\max}.
    \end{cases}
  \end{equation*}

  The \emph{initial form} of $f$ with respect to a weight vector $w\in\RR^n$ is then defined as
  \[ \initial_w(f)\coloneqq \sum_{\substack{\trop(c_\alpha x^\alpha)(w)\\=\trop(f)(w)}} \overline{t^{-\val(c_\alpha)} c_\alpha}\cdot x^\alpha\in \mathfrak K[x], \]
  where $\overline{(\cdot)}\colon R\twoheadrightarrow\mathfrak K$ denotes the quotient map.

  The \emph{initial ideal} of an ideal $I\unlhd K[x]$ with respect to a weight vector $w\in\RR^n$ is given by
  \[ \initial_w(I)\coloneqq\langle \initial_w(f)\mid f\in I\rangle\unlhd \mathfrak K[x]. \]

  A \emph{Gr\"obner basis} of an ideal $I\unlhd K[x]$ with respect to a weight vector $w\in\RR^n$ is a finite set $G\subseteq I$ such that $\initial_w(I)=\langle\initial_w(g)\mid g\in G\rangle$.  Note that, any Gr\"obner basis of $I$ is a generating set provided $I$ is homogeneous \cite[Remark 2.4.4]{MS2015}.
\end{definition}

\begin{example}\label{ex:groebnerBases}
  In \OSCAR, tropical features over a valued field $K$ generally require specifying a map to a tropical semiring $\nu\colon K^\ast\rightarrow\TT$ satisfying
  \begin{enumerate}
  \item $\nu(a\cdot b)=\nu(a)+\nu(b)$,
  \item $\nu(a+b)\geq\min(\nu(a),\nu(b))$ (in the previously defined ordering on $\TT$).
  \end{enumerate}
  Most commonly, $\nu(c)=\val(c)$ for $\TT=\TT_{\min}$ and $\nu(c)=-\val(c)$ for $\TT=\TT_{\max}$. Hence, $\nu$ encodes both the valuation as well as the choice of min- or max-convention.  Currently, only specific maps are supported such as:
  \begin{enumerate}
  \item trivial valuations on any field:
\begin{minted}{jlcon}
julia> nu = tropical_semiring_map(GF(3))
Map into Min tropical semiring encoding the trivial valuation on Finite field of characteristic 3

julia> nu.([0,1,2])
3-element Vector{TropicalSemiringElem{typeof(min)}}:
 TODOINFINITY
 (0)
 (0)
\end{minted}
  \item $p$-adic valuations on $\QQ$:
\begin{minted}{jlcon}
julia> nu = tropical_semiring_map(QQ,3,max)
Map into Max tropical semiring encoding the 3-adic valuation on Rational field

julia> nu.([1//3,1,3])
3-element Vector{TropicalSemiringElem{typeof(max)}}:
 (1)
 (0)
 (-1)
\end{minted}
  \item $t$-adic valuations on any rational function field:
\begin{minted}{jlcon}
julia> Ft,t = RationalFunctionField(GF(3),"t")
(Rational function field over GF(3), t)

julia> nu = tropical_semiring_map(Ft,t,max)
Map into Max tropical semiring encoding the t-adic valuation on Rational function field over GF(3)

julia> nu.([t^(-1),1,t])
3-element Vector{TropicalSemiringElem{typeof(max)}}:
 (1)
 (0)
 (-1)
\end{minted}
  \end{enumerate}

  As in \cref{ex:tropicalSemiring}, the default codomain of tropical semiring maps is $\TT_{\min}$.  The tropical semiring map is usually passed as input, for example:
\begin{minted}{jlcon}
julia> R,(x1,x2,x3,x4) = QQ["x1","x2","x3","x4"];

julia> I = ideal([x1+2*x2-3*x3, 3*x2-4*x3+5*x4]);

julia> nu_2 = tropical_semiring_map(QQ,2);

julia> GBI = groebner_basis(I,nu_2,[0,0,0,0])
2-element Vector{QQMPolyRingElem}:
 3*x2 - 4*x3 + 5*x4
 x1 + 2*x2 - 3*x3

julia> inI = initial(I,nu_2,[0,0,0,0])
ideal(x2 + x4, x1 + x3)

julia> GBI = groebner_basis(I,nu_2,[1,0,0,1])
2-element Vector{QQMPolyRingElem}:
 -x1 + x2 - x3 + 5*x4
 -3*x1 + x3 + 10*x4

julia> inI = initial(I,nu_2,[1,0,0,1])
ideal(x2 + x3, x3)

julia> coefficient_ring(inI)
Finite field of characteristic 2
\end{minted}

  The code above shows \cite[Example 2.4.3]{MS:2015}, Gr\"obner bases of the ideal $I\coloneqq \langle x_1+2x_2-3x_3,3x_2-4x_3+5x_4\rangle\unlhd\QQ[x_1,\dots,x_4]$ with respect to the $2$-adic valuation and the two weight vectors $(0,0,0,0)\in\Trop(V(I))$ and $(1,0,0,1)\notin\Trop(V(I))$.  Note that the initial ideals in tropical geometry are defined over the residue field $\mathbb{F}_2$.
\end{example}

\subsection{Tropical varieties}

Next, we recall the notion of tropical varieties.  While \cite{MS2015} considers tropicalizations of very affine varieties $V(I)$, we consider tropicalizations of (Laurent) polynomial ideal $I$.  Consequently, our tropicalizations are weighted polyhedral complexes instead of supports thereof.  The reason for this is two-fold:
\begin{enumerate}
\item From \cref{sec:genericRootCount} onwards, we are interested in the number of points in $V(I)$ counted with multiplicity, so the tropical multiplicities matter.
\item For \OSCAR, it is much easier to work with concrete polyhedral complexes rather than supports thereof.
\end{enumerate}
Essentially, our definition of tropical varieties entails part of \cite[Structure Theorem 3.3.5]{MS2015}.

\begin{definition}\label{def:groebnerComplex}
  Let $I\unlhd K[x]$ be a homogeneous ideal.  Then the \emph{Gr\"obner polyhedron} of $I$ around a weight vector $w\in\RR^n$ is defined to be
  \[ C_w(I)\coloneqq \overline{\{ v\in\RR^n\mid \initial_v(I)=\initial_w(I)\}}, \]
  and the set of all Gr\"obner polyhedra is called the \emph{Gr\"obner complex} $\Sigma(I)\coloneqq\{C_w(I)\mid w\in\RR^n\}$.  As the name suggest, Gr\"obner polyhedra are convex polyhedra and the Gr\"obner complex is a finite polyhedral complex \cite[Theorem 2.5.3]{MS2015}.
\end{definition}

\begin{definition}\label{def:tropicalVariety}
  The \emph{tropical variety} of an homogeneous ideal $I\unlhd K[x]$ is defined to be the following polyhedral complex consisting of all Gr\"obner cones whose initial ideals are monomial free:
  \[ \Trop(I)\coloneqq \{C_w(I)\mid \initial_w(I) \text{ monomial-free}\} \]
  The \emph{tropical variety} of an arbitrary ideal $I\unlhd K[x]$ is defined to be
  \[ \Trop(I)\coloneqq \Big\{C_w(I^h) \cap \{e_0=0\} \mid C_w(I^h)\in\Trop(I^h) \Big\}, \]
  where $I^h\unlhd K[x_0,\dots,x_n]$ denotes the homogenization of $I$ in $x_0$, whose Gr\"obner polyhedra $C_w(I^h)$ lie in $\RR^{n+1}$ with standard basis vectors $e_0,\dots,e_n$, and we naturally identify the subspace $\{e_0=0\}\subseteq\RR^{n+1}$ with $\RR^n$.  Note that both definitions coincide for homogeneous ideals.

  Moreover, there is a canonical way to assign multiplicities to the maximal polyhedra of $\Trop(I)$ via the initial ideals \cite[Definition 3.4.3]{MS2015} that make $\Trop(I)$ a balanced polyhedral complex if $I$ is prime \cite[Theorem 3.3.5]{MS2015}.  For the sake of brevity, we omit their general definition and instead introduce them only for the special classes of tropical varieties that are of immediate interest to us below.
\end{definition}

\begin{remark}\label{rem:tropicalVariety}
  The umbrella term tropical variety extends well beyond the tropicalization of ideals covered in \cref{def:tropicalVariety}.  Most notably, there is an incredibly rich theory on the moduli of tropical curves, abstract graphs without a fixed embedding into some $\RR^n$, whose significance extends well beyond the area of tropical geometry itself \cite{Markwig2020}.

  \OSCAR uses parameters to distinguish between embedded and abstract tropical varieties, as well as min- and max-tropical varieties.  Moreover, there are four distinct types of tropical varieties, see Examples \ref{ex:tropicalHypersurface} to \ref{ex:tropicalVariety}.  This is because some functions only make sense for some types (e.g., tropical Pluecker vectors for tropical linear spaces), while other functions can be aggressively optimised depending on the type (e.g., stable intersection of tropical hypersurfaces).  Embedded tropical varieties are always weighted polyhedral complexes.  Abstract tropical varieties depend on the type and may not be defined for all four types.  \end{remark}

\begin{example}\label{ex:tropicalVariety}\
  Objects of type \texttt{TropicalVariety} represent the most general type of tropical varieties and currently need to be embedded, abstract general tropical varieties are presently not supported.  Embedded \texttt{TropicalVariety}s are weighted polyhedral complexes, that need not be balanced or pure.  They can be constructed from polynomial ideals:
\begin{minted}{jlcon}
julia> K,t = RationalFunctionField(GF(101),"t");

julia> nu = tropical_semiring_map(K,t);

julia> R,(x,y,z) = K["x","y","z"];

julia> I = intersect(ideal([x+y+z+1,2*x+11*y+23*z+31]),
                     ideal([t^3*x*y*z-1]));

julia> TropV = tropical_variety(I,nu)
2-element Vector{TropicalVariety}:
 min tropical variety
 min tropical variety

julia> dim.(TropV)
2-element Vector{Int64}:
 2
 1
\end{minted}
  The command \texttt{tropical\_variety(::MPolyIdeal)} returns an array of tropical varieties, one for each primary factor. The example above consists of a two-dimensional tropicalized binomial variety and a one-dimensional tropicalized linear space, see \cref{fig:tropicalVariety}.  It is currently limited to special types of ideals, and an extension to general ideals is in the works.

  Objects of type \texttt{TropicalVariety} also arise naturally from the other types, e.g., as the (stable) intersection of two tropical hypersurfaces, see \cref{ex:stableIntersection}.
\end{example}

\begin{figure}[h]
  \centering
  \begin{tikzpicture}[x={(240:1cm)},y={(0:1cm)},z={(90:1cm)}]
    \fill[blue!20]
    ($(0,-1,-0.25)+(-1,1,0)+0.75*(1,1,-2)$)
    -- ($(0,-1,-0.25)-(-1,1,0)+0.75*(1,1,-2)$)
    -- ($(0,-1,-0.25)-(-1,1,0)-0.75*(1,1,-2)$)
    -- node[sloped,above,blue!20!black] {$\Trop(t^3xyz-1)$} ($(0,-1,-0.25)+(-1,1,0)-0.75*(1,1,-2)$)
    -- cycle;

    \draw[thick,->]
    (0,0,0) -- ++(2,0,0) node[anchor=north] {$e_x$};
    \draw[thick,->]
    (0,0,0) -- node[anchor=north west,yshift=-1mm] {$\Trop(\langle x+y+z+1,$} node[anchor=north west,yshift=-6mm] {$\phantom{\Trop(\langle} 2x+11y+23z+31\rangle)$} ++(0,2,0) node[anchor=west] {$e_y$};
    \draw[thick,->]
    (0,0,0) -- ++(0,0,2) node[anchor=south] {$e_z$};
    \draw[thick]
    (0,0,0) -- ++(-2,-2,-2);
    \draw[thick,dashed]
    (0,0,0)++(-2,-2,-2) -- ++(-3,-3,-3);
    \draw[thick,->]
    (0,0,0)++(-5,-5,-5) -- ++(-1,-1,-1) node[anchor=east] {$-e_x-e_y-e_z$};
    \coordinate (pt) at (-2,-2,-2);
    \fill[blue!60!black] (pt) circle (2.5pt);
    \node[anchor=south,blue!60!black,xshift=-5mm,yshift=0mm] at (pt) {$(-1,-1,-1)$};
    \fill[black] (0,0,0) circle (2.5pt);
  \end{tikzpicture}
  \caption{The tropical variety in \cref{ex:tropicalVariety}.}
  \label{fig:tropicalVariety}
\end{figure}
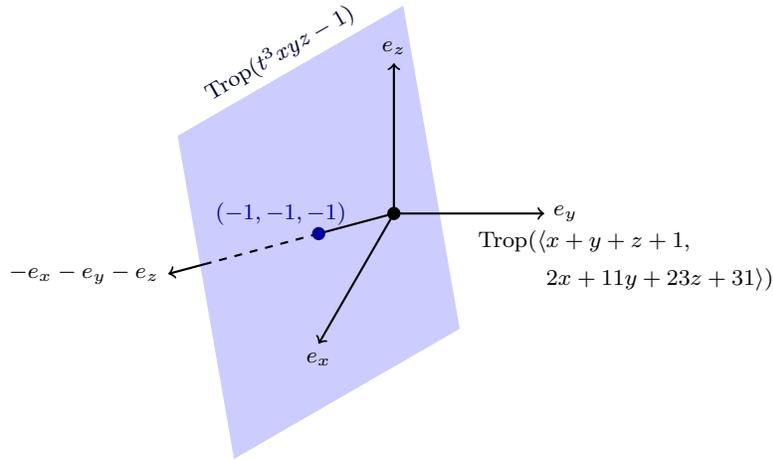

\begin{example}\label{ex:tropicalHypersurface}
  If $I$ is generated by a single polynomial $f=\sum_{\alpha\in\ZZ^n_{\geq 0}}c_\alpha x^\alpha$, $c_\alpha\in K$, then $\Trop(f)\coloneqq\Trop(I)$ is a \emph{tropical hypersurface} \cite[Section 3.1]{MS2015}.  Its cells are dual to the regular subdivision of its Newton polytope $\conv\{\alpha\mid c_\alpha\neq 0\}$ induced by the valuation of its coefficients \cite[Lemma 3.4.6]{MS2015}.  Its maximal cells are dual to edges and their multiplicities are the lattice lengths of the edges.  \cref{fig:tropicalHypersurface} illustrates the Newton subdivisions and max-tropical hypersurfaces of
  \begin{align*}
    f \coloneqq & t^3+x+t^2\cdot y+x(x^2+y^2)\quad\text{and}\\
    g \coloneqq & t^4+t^4\cdot x+t^2\cdot y+y(x^2+y^2) \in \CC(t)[x,y].
  \end{align*}

  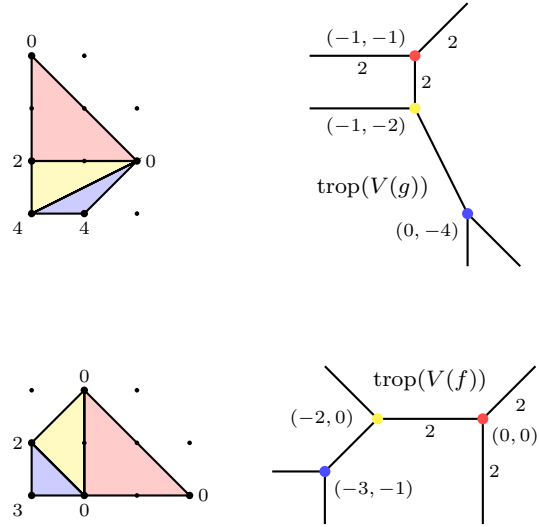
\begin{figure}[t]
    \centering
    \begin{tikzpicture}
      \node[anchor=south west] (NewtonF) at (0,0)
      {
        \begin{tikzpicture}[scale=0.7]
          \fill[blue!20] (0,0) -- (1,0) -- (0,1);
          \fill[yellow!30] (1,0) -- (1,2) -- (0,1);
          \fill[red!20] (1,0) -- (3,0) -- (1,2);
          \draw[thick]
          (0,0) -- (1,0) -- (0,1) -- cycle
          (1,0) -- (1,2) -- (0,1) -- cycle
          (1,0) -- (3,0) -- (1,2) -- cycle;
          \fill
          (0,0) circle (2pt)
          (1,0) circle (2pt)
          (2,0) circle (1.25pt)
          (3,0) circle (2pt)
          (0,1) circle (2pt)
          (1,1) circle (1.25pt)
          (2,1) circle (1.25pt)
          (3,1) circle (1.25pt)
          (0,2) circle (1.25pt)
          (1,2) circle (2pt)
          (2,2) circle (1.25pt);
          \node[anchor=north east,font=\scriptsize] at (0,0) {$3$};
          \node[anchor=north,font=\scriptsize] at (1,0) {$0$};
          \node[anchor=east,font=\scriptsize] at (0,1) {$2$};
          \node[anchor=west,font=\scriptsize] at (3,0) {$0$};
          \node[anchor=south,font=\scriptsize] at (1,2) {$0$};
        \end{tikzpicture}
      };
      \node[anchor=west,xshift=5mm] (TropF) at (NewtonF.east)
      {
        \begin{tikzpicture}[scale=0.7]
          \draw[thick]
          (0,0) -- node[anchor=north west,font=\scriptsize,inner sep=2pt] {$2$} ++(1,1)
          (0,0) -- node[right,font=\scriptsize,inner sep=2pt] {$2$} ++(0,-2)
          (0,0) -- node[below,font=\scriptsize,inner sep=2pt] {$2$} (-2,0)
          (-2,0) -- ++(-1,1)
          (-2,0) -- (-3,-1)
          (-3,-1) -- ++(-1,0)
          (-3,-1) -- ++(0,-1);
          \fill[red!70] (0,0) circle (3pt);
          \fill[yellow!80] (-2,0) circle (3pt);
          \fill[blue!70] (-3,-1) circle (3pt);
          \node[anchor=north west,font=\scriptsize] at (0,0) {$(0,0)$};
          \node[anchor=east,xshift=-2mm,font=\scriptsize] at (-2,0) {$(-2,0)$};
          \node[anchor=north west,font=\scriptsize] at (-3,-1) {$(-3,-1)$};
          \node[anchor=south,yshift=2.5mm] at (-1,0) {$\trop(V(f))$};
        \end{tikzpicture}
      };
      \node[anchor=south west] (NewtonG) at (0,3.75)
      {
        \begin{tikzpicture}[scale=0.7]
          \fill[blue!20] (0,0) -- (1,0) -- (2,1);
          \fill[yellow!30] (0,0) -- (2,1) -- (0,1);
          \fill[red!20] (2,1) -- (0,1) -- (0,3);
          \draw[thick]
          (0,0) -- (1,0) -- (2,1) -- cycle
          (0,0) -- (2,1) -- (0,1) -- cycle
          (2,1) -- (0,1) -- (0,3) -- cycle;
          \fill
          (0,0) circle (2pt)
          (1,0) circle (2pt)
          (2,0) circle (1.25pt)
          (0,1) circle (2pt)
          (1,1) circle (1.25pt)
          (2,1) circle (2pt)
          (0,2) circle (1.25pt)
          (1,2) circle (1.25pt)
          (2,2) circle (1.25pt)
          (0,3) circle (2pt)
          (1,3) circle (1.25pt);
          \node[anchor=north east,font=\scriptsize] at (0,0) {$4$};
          \node[anchor=north,font=\scriptsize] at (1,0) {$4$};
          \node[anchor=east,font=\scriptsize] at (0,1) {$2$};
          \node[anchor=west,font=\scriptsize] at (2,1) {$0$};
          \node[anchor=south,font=\scriptsize] at (0,3) {$0$};
        \end{tikzpicture}
      };
      \node[anchor=west,xshift=17mm] (TropG) at (NewtonG.east)
      {
        \begin{tikzpicture}[scale=0.7]
          \draw[thick]
          (-1,-1) -- node[anchor=north west,font=\scriptsize,inner sep=2pt] {$2$} ++(1,1)
          (-1,-1) -- node[below,font=\scriptsize,inner sep=2pt] {$2$} ++(-2,0)
          (-1,-1) -- node[right,font=\scriptsize,inner sep=2pt] {$2$} (-1,-2)
          (-1,-2) -- ++(-2,0)
          (-1,-2) -- (0,-4)
          (0,-4) -- ++(0,-1)
          (0,-4) -- ++(1,-1);
          \fill[red!70] (-1,-1) circle (3pt);
          \fill[yellow!80] (-1,-2) circle (3pt);
          \fill[blue!70] (0,-4) circle (3pt);
          \node[anchor=south east,font=\scriptsize] at (-1,-1) {$(-1,-1)$};
          \node[anchor=north east,font=\scriptsize] at (-1,-2) {$(-1,-2)$};
          \node[anchor=north east,font=\scriptsize] at (0,-4) {$(0,-4)$};
          \node[anchor=east,xshift=-4mm] at (0,-3.5) {$\trop(V(g))$};
        \end{tikzpicture}
      };
    \end{tikzpicture}
    \caption{Two tropical hypersurfaces and their Newton subdivisions}\label{fig:tropicalHypersurface}
  \end{figure}

  In \OSCAR, objects of type \texttt{TropicalHypersurface} need to be embedded and are objects of type \texttt{TropicalVariety} that are of codimension one. They can constructed using
  \begin{enumerate}
  \item polynomials over valued fields:
\begin{minted}{jlcon}
julia> K,t = RationalFunctionField(QQ,"t");

julia> nu = tropical_semiring_map(K,t,max);

julia> R,(x,y) = K["x","y"];

julia> f = t^3+x+t^2*y+x*(x^2+y^2)
x^3 + x*y^2 + x + t^2*y + t^3

julia> TropH = tropical_hypersurface(f,nu)
max tropical hypersurface

julia> vertices_and_rays(TropH) # 1,2,3 are vertices, rest are rays
7-element SubObjectIterator{Union{PointVector{QQFieldElem}, RayVector{QQFieldElem}}}:
 [0, -1]
 [-2, 0]
 [-3, -1]
 [0, 0]
 [1, 1]
 [-1, 1]
 [-1, 0]

julia> maximal_polyhedra(IncidenceMatrix,TropH)
7×7 IncidenceMatrix
[4, 5]
[1, 4]
[2, 4]
[2, 6]
[2, 3]
[1, 3]
[3, 7]
\end{minted}
    Note above that applying \texttt{typeof} to the seven entries of the output of \texttt{vertices\_and\_rays} reveals that the first three are vertices whilst the latter four are ray generators.  Hence \texttt{[1,3]} in the output of \texttt{maximal\_polyhedra} is the cone \texttt{}
  \item polynomials over tropical semirings:
\begin{minted}{jlcon}
jjulia> tropf = tropical_polynomial(f,nu)
x^3 + x*y^2 + x + (-2)*y + (-3)

julia> tropical_hypersurface(tropf)
max tropical hypersurface
\end{minted}
  \item subdivision of points:
\begin{minted}{jlcon}
julia> points = matrix(ZZ,collect(exponents(tropf)))
[3   0]
[1   2]
[1   0]
[0   1]
[0   0]

julia> heights = QQ.(collect(coefficients(tropf)))
5-element Vector{QQFieldElem}:
 0
 0
 0
 -2
 -3

julia> Delta = subdivision_of_points(points,heights)
Subdivision of points in ambient dimension 2

julia> tropical_hypersurface(Delta,max)
max tropical hypersurface
\end{minted}
  \end{enumerate}
\end{example}

\begin{example}\label{ex:tropicalLinearSpace}\
  If $I$ is generated by linear forms, then $\Trop(I)$ is a (realizable) \emph{tropical linear space} \cite[Section 4.4]{MS2015} \cite[Section 10.4 and 10.5]{Joswig21}.  It is dual to a matroid subdivision of the hypersimplex, and the support of its recession fan is the support of a Bergman fan.  Its maximal cells are all of multiplicity $1$.  Note that \OSCAR returns the same polyhedral structure as \textsc{polymake}, which is the one computed by Rincon's algorithm using fundamental circuits, where cyclic flats give rise to rays \cite{Rincon13}.

  In \OSCAR, objects of type \texttt{TropicalLinearSpace} need to be embedded and can be thought of objects of type \texttt{TropicalVariety} that are of degree $1$ with respect to the intersection product in \cref{def:stableIntersection}.  A tropical linear space may be constructed from
  \begin{enumerate}
  \item ideals over valued fields:
\begin{minted}{jlcon}
julia> R,(w,x,y,z) = QQ["w","x","y","z"];

julia> nu = tropical_semiring_map(QQ);

julia> I = ideal([w+x+y+z,w+2*x+5*y+11*z]);

julia> TropL = tropical_linear_space(I,nu)
min tropical linear space

julia> vertices_and_rays(TropL)
5-element SubObjectIterator{Union{PointVector{QQFieldElem}, RayVector{QQFieldElem}}}:
 [0, 0, 0, 0]
 [0, -1, -1, -1]
 [0, 1, 0, 0]
 [0, 0, 1, 0]
 [0, 0, 0, 1]

julia> maximal_polyhedra(IncidenceMatrix,TropL)
4×5 IncidenceMatrix
[1, 2]
[1, 3]
[1, 4]
[1, 5]
\end{minted}
  \item matrices over valued fields or tropical semirings:
\begin{minted}{jlcon}
julia> A = matrix(QQ,[1 1 1 0; 1 1 0 1])
[1   1   1   0]
[1   1   0   1]

julia> TropL1 = tropical_linear_space(A, nu) # not Stiefel
min tropical linear space

julia> tropA = nu.(A)
2×4 Matrix{TropicalSemiringElem{typeof(min)}}:
 (0)  (0)  (0)  TODOINFINITY
 (0)  (0)  TODOINFINITY    (0)

julia> TropL2 = tropical_linear_space(tropA) # not equal TropL1
min tropical linear space
\end{minted}
    In the example above \texttt{TropL1} is not the same as \texttt{TropL2}, as the tropicalised minors of a matrix need not be the minors of a tropicalised matrix.  In fact, \texttt{TropL1} cannot be constructed from a tropical matrix.  Tropical linear spaces which can be constructed from tropical matrices are known as Stiefel topical linear spaces \cite{FinkRincon15}.
  \item Pl\"ucker vectors over valued fields or tropical semirings:
\begin{minted}{jlcon}
julia> plueckerIndices = [[1,2],[1,3],[1,4],[2,3],[2,4],[3,4]];

julia> algebraicPlueckerVector = [0,-1,-1,1,1,1];

julia> TropL1 = tropical_linear_space(plueckerIndices,
                                      algebraicPlueckerVector,
                                      nu)
min tropical linear space

julia> tropicalPlueckerVector = nu.(algebraicPlueckerVector);

julia> TropL2 = tropical_linear_space(plueckerIndices,
                                      tropicalPlueckerVector)
min tropical linear space
\end{minted}
  \end{enumerate}
\end{example}

\begin{example}\label{ex:tropicalCurve}
  If $\dim(I)=1$, then $\Trop(I)$ is an (embedded) tropical curve.  In \OSCAR, objects of type \texttt{TropicalCurve} are either weighted polyhedral complexes of dimension $1$ if embedded or weighted graphs if abstract.  Embedded \texttt{TropicalCurve}s can be constructed by converting other tropical varieties.
\begin{minted}{jlcon}
julia> R,(x,y) = QQ["x","y"];

julia> nu = tropical_semiring_map(QQ);

julia> f = 1+x+y+x*(x^2+y^2);

julia> TropH = tropical_hypersurface(f,nu)
min tropical hypersurface

julia> TropC = tropical_curve(TropH)
Embedded min tropical curve
\end{minted}

  Abstract \texttt{TropicalCurve} can be constructed from graphs (if no weights are specified, they are set to $1$ by default).
\begin{minted}{jlcon}
julia> G = dualgraph(cube(3))
Undirected graph with 6 nodes and the following edges:
(3, 1)(3, 2)(4, 1)(4, 2)(5, 1)(5, 2)(5, 3)(5, 4)(6, 1)(6, 2)(6, 3)(6, 4)

julia> tropical_curve(G)
Abstract min tropical curve
\end{minted}
\end{example}

\subsection{Stable intersections}
Finally, we recall tropical stable intersection and tropical intersection products, which is a central concept for \cref{sec:genericRootCount} onwards.

\begin{definition}\label{def:stableIntersection}
  Let $\Sigma_1$, $\Sigma_2$ be two balanced polyhedral complexes in $\RR^n$.  Their \emph{stable intersection} is defined to be
  \[ \Sigma_1\stsect\Sigma_2 \coloneqq \Big\{ \sigma_1\cap\sigma_2\mid \sigma_1\in\Sigma_1,\sigma_2\in\Sigma_2,\dim(\sigma_1+\sigma_2)=n \Big\} \]
  with multiplicities
  \[ \mult_{\Sigma_1\stsect\Sigma_2}(\sigma_1\cap\sigma_2) \coloneqq \sum_{\tau_1,\tau_2}\mult_{\Sigma_1}(\tau_1)\mult_{\Sigma_2}(\tau_2)[N:N_{\tau_1}+N_{\tau_2}], \]
  where the sum is over all $\tau_1\in\mathrm{star}_{\Sigma_1}(\sigma_1\cap\sigma_2)$ and $\tau_2\in\mathrm{star}_{\Sigma_2}(\sigma_1\cap\sigma_2)$ with $\tau_1\cap (\tau_2+v)\neq\emptyset$ for some fixed generic $v\in\RR^n$.  It is either empty, or again a balanced polyhedral complex of codimension $\codim(\Sigma_1)+\codim(\Sigma_2)$ \cite[Theorem 3.6.10]{MS2015}.  Alternatively, its support can also be defined as the intersection under generic perturbation \cite[Proposition 3.6.12]{MS2015}.

  If $\Sigma_1,\dots,\Sigma_r$ are balanced polyhedral complexes in $\RR^n$ of complementary dimension, i.e., $\sum_{i=1}^r\codim(\Sigma_i)=n$, then their \emph{tropical intersection number} $\prod_{i=1}^r\Sigma_i$ is the number of points in their stable intersection $\wedge_{i=1}^r \Sigma_i$ counted with multiplicity.
\end{definition}

\begin{example}\label{ex:stableIntersection}
  \cref{fig:stableIntersection} illustrates the stable intersection of their tropical hypersurfaces from \cref{ex:tropicalHypersurface}.  It can be computed using \texttt{stable\_intersection}:
\begin{minted}{jlcon}
julia> K,t = RationalFunctionField(QQ,"t");

julia> nu = tropical_semiring_map(K,t,max);

julia> R,(x,y) = K["x","y"];

julia> f = t^3+x+t^2*y+x*(x^2+y^2)
x^3 + x*y^2 + x + t^2*y + t^3

julia> g = t^4+t^4*x+t^2*y+y*(x^2+y^2)
x^2*y + t^4*x + y^3 + t^2*y + t^4

julia> TropHf = tropical_hypersurface(f,nu)
max tropical hypersurface

julia> TropHg = tropical_hypersurface(g,nu)
max tropical hypersurface

julia> TropV = intersect_stably(TropHf,TropHg)
max tropical variety

julia> vertices(TropV)
4-element SubObjectIterator{PointVector{QQFieldElem}}:
 [0, 0]
 [0, -4]
 [-3, -1]
 [-3, -2]

julia> multiplicities(TropV) # WARNING: not same ordering as above
Dict{Polyhedron{QQFieldElem}, ZZRingElem} with 4 entries:
  Polyhedron in ambient dimension 2 => 2
  Polyhedron in ambient dimension 2 => 4
  Polyhedron in ambient dimension 2 => 2
  Polyhedron in ambient dimension 2 => 1
\end{minted}
\end{example}

\begin{figure}
  \centering
  \begin{tikzpicture}
    \node at (0,0)
    {
      \begin{tikzpicture}[scale=0.7]
        \draw[blue!50!black,thick]
        (0,0) -- ++(1,1)
        (0,0) -- ++(0,-5)
        (0,0) -- (-2,0)
        (-2,0) -- ++(-1,1)
        (-2,0) -- (-3,-1)
        (-3,-1) -- ++(-1,0)
        (-3,-1) -- ++(0,-4);
        \fill[blue!50!black]
        (0,0) circle (2pt)
        (-2,0) circle (2pt)
        (-3,-1) circle (2pt);

        \begin{scope}[xshift=-5mm,yshift=5mm]
          \draw[red!70!black,thick]
          (-1,-1) -- ++(1.5,1.5)
          (-1,-1) -- ++(-2.5,0)
          (-1,-1) -- (-1,-2)
          (-1,-2) -- ++(-2.5,0)
          (-1,-2) -- (0,-4)
          (0,-4) -- ++(0,-1.5)
          (0,-4) -- ++(1.5,-1.5);
          \fill[red!70!black]
          (-1,-1) circle (2pt)
          (-1,-2) circle (2pt)
          (0,-4) circle (2pt);
        \end{scope}

        \fill[white,draw=black]
        (-1,0) circle (2.5pt)
        (-2.5,-0.5) circle (2.5pt)
        (-3,-1.5) circle (2.5pt)
        (0,-4) circle (2.5pt);

        \draw[<-,thick,violet]
        (-1.5,-1) -- node[anchor=south west,font=\scriptsize,inner sep=1pt] {$v$} ++(0.5,-0.5);
      \end{tikzpicture}
    };
    \draw[->,violet,dashed] (2.25,0) -- node[anchor=north] {$v\rightarrow 0$} (3.75,0);
    \node at (6,0)
    {
      \begin{tikzpicture}[scale=0.7]
        \draw[blue!50!black,thick]
        (0,0) -- ++(0,-4)
        (0,0) -- (-2,0)
        (-2,0) -- ++(-1,1)
        (-2,0) -- (-3,-1)
        (-3,-1) -- ++(0,-4);
        \draw[red!70!black,thick]
        (-1,-1) -- ++(1,1)
        (-1,-1) -- ++(-2,0)
        (-1,-1) -- (-1,-2)
        (-1,-2) -- ++(-3,0)
        (-1,-2) -- (0,-4)
        (0,-4) -- ++(1,-1);
        \draw[blue!50!black,very thick,dash pattern= on 3pt off 5pt,dash phase=4pt]
        (0,0) -- ++(1,1)
        (-3,-1) -- ++(-1,0)
        (0,-4) -- ++(0,-1);
        \draw[red!70!black,very thick,dash pattern= on 3pt off 5pt]
        (0,0) -- ++(1,1)
        (-3,-1) -- ++(-1,0)
        (0,-4) -- ++(0,-1);
        \fill[blue!50!black]
        (-2,0) circle (2pt);
        \fill[red!70!black]
        (-1,-1) circle (2pt)
        (-1,-2) circle (2pt);
        \fill[white,draw=black]
        (0,0) circle (2.5pt)
        (-3,-1) circle (2.5pt)
        (-3,-2) circle (2.5pt)
        (0,-4) circle (2.5pt);
      \end{tikzpicture}
    };
  \end{tikzpicture}
  \caption{Stable intersection (white vertices, right) of two tropical plane curves (red and blue, right) arising as a limit of their transverse intersection (white vertices, left) under perturbation by a generic $v$ as $v\rightarrow 0$.}\label{fig:stableIntersection}
\end{figure}
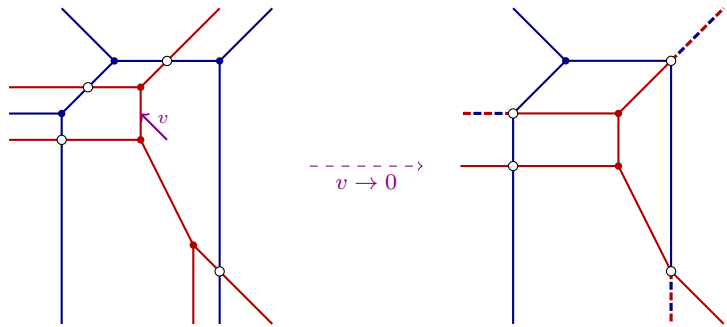

The significance of stable intersections is the fact that they are the ``expected'' tropicalisations when taking the sum of two ideals.  This is for example captured in the following theorem:

\begin{theorem}\label{thm:transverseIntersection}
  Let $I_1,I_2\unlhd \CC[x]$ be complete intersections.  Suppose that $\Trop(I_1)$ and $\Trop(I_2)$ intersect transversely.  Then
  \[ \Trop(I_1 + I_2)=\Trop(I_1)\stsect \Trop(I_2), \]
  where the equality only holds up to refinement of weighted polyhedral complexes.
\end{theorem}
\begin{proof}
  The statement is a consequence of \cite[Corollary 5.1.2]{OssermanPayne13} and \cite[Corollary 5.1.3]{OssermanPayne13} by Osserman and Payne.  The first result shows the equality set-theoretically, while the second result shows that the multiplicities coincide.  Both results require $I_1$ and $I_2$ to be Cohen-Macaulay, which is implied from $I_1$ and $I_2$ being complete intersections.
\end{proof}

\cref{thm:transverseIntersection} is a generalisation of the Transverse Intersection Theorem \cite[Theorem 3.4.12]{MS2015} which relies on \cite[Corollary 5.1.2]{OssermanPayne13}.  From now onward, we will generally consider tropical varieties up to refinement, which does not change the tropical intersection number in \cref{def:stableIntersection}.

\section{Tropical geometry of generic root counts}\label{sec:genericRootCount}

In this section, we recall some basic notions and a result of \cite{HelminckRen22}.

\begin{convention}\label{con:parametrizedPolynomialSystem}
	In addition to Convention~\ref{con:valuedField}, we further fix an affine space $K^m$, which we refer to as the \emph{parameter space}. Let $K[a]\coloneqq K[a_1,\dots,a_m]$ be its coordinate ring and $K(a)\coloneqq K(a_1,\dots,a_m)$ be the field of fractions thereof.

	We will refer to points $P\in K^m$ as \emph{choices of parameters}, elements $f\in K[a][x^\pm]$ as \emph{parametrised (Laurent) polynomials}, ideals $I\subseteq K[a][x^\pm]$ as \emph{parametrised (Laurent) polynomial ideals}, and finite generating sets\linebreak $\{f_1,\dots,f_k\}\subseteq I$ as \emph{parametrised (Laurent) polynomial systems}.
\end{convention}

\begin{definition}
	Let $P\in K^m$ be a choice of parameters.  For any parametrised polynomial $f\in K[a][x^\pm]$, say $f=\sum_{\alpha\in\ZZ^n}c_\alpha x^\alpha$ with $c_\alpha\in K[a]$, we define the \emph{specialisation} of $f$ at $P$ to be
	\begin{equation*}
		f_{P}=\sum_{\alpha\in\ZZ^n}c_\alpha(P)x^\alpha\in K[x^\pm].
	\end{equation*}
	For any parametrised polynomial ideal $I\unlhd K[a][x^\pm]$, we define the \emph{specialisation} of $I$ at $P$ to be
	\begin{equation*}
		I_{P}=\langle f_{P}\mid f\in I \rangle\unlhd K[x^\pm].
	\end{equation*}
	The \emph{root count} of $I$ at $P$ is the vector space dimension
  \begin{equation*}
    \ell_{I,P}\coloneqq \dim_K(K[x^\pm]/I_P)\in\mathbb{N}\cup\{\infty\}.
  \end{equation*}
\end{definition}

\begin{definition}
	The \emph{generic specialisation} of a parametrised polynomial ideal $I\unlhd K[a][x^\pm]$ is the ideal $I_{K(a)}\unlhd K(a)[x^\pm]$ generated by $I$ under the inclusion $K[a][x^\pm]\subseteq K(a)[x^\pm]$.
	The \emph{generic root count} of $I$ is the vector space dimension
  \[ \ell_I\coloneqq\dim_{K(a)}(K(a)[x^\pm]/I_{K(a)})\in\mathbb{N}\cup\{\infty\}. \]
We say $I_1,\dots,I_r\unlhd K[a][x^\pm]$ are \emph{generically of complementary dimension}, if
  \[ \codim(I_{1,K(a)})+\dots +\codim(I_{r,K(a)}) = n, \]
  where $\codim(\cdot)$ denotes the Krull-codimension.
\end{definition}

\begin{definition}
  Recall that the torus $(K^\ast)^n$ is an algebraic group with group multiplication
  \[ (K^\ast)^n\times (K^\ast)^n\rightarrow (K^\ast)^n, \quad (\underbrace{(t_1,\dots,t_n)}_{\eqqcolon t},\underbrace{(s_1,\dots,s_n)}_{\eqqcolon s})\mapsto \underbrace{(t_1s_1,\dots,t_ns_n)}_{\eqqcolon t\cdot s}. \]
  A parametrised polynomial ideal $I\unlhd K[a][x^\pm]$ is \emph{torus-equivariant}, if there is a matching torus action on the parameter space
  \[ (K^\ast)^n\times K^m\rightarrow K^m, \quad (t,P)\mapsto t\cdot P \quad\text{such that } V(I_{t\cdot P}) = t\cdot V(I_P). \]
  Moreover, several parametrised polynomial ideals $I_1,\dots,I_r\unlhd K[a][x^\pm]$ are \emph{parametrically independent}, if their generators have distinct parameters, i.e., $I_i=\langle F_i\rangle$ with $F_i\subseteq K[a_j\mid j\in A_i][x^\pm]$ for some partition $[m]=\coprod_{i=1}^rA_i$.
\end{definition}

\begin{example}
  The principal ideal $I\unlhd K[a_1,\dots,a_5][x_1^\pm,x_2^\pm]$ generated by
  \[ f(a;x) \coloneqq a_1x_1^2+a_2x_1x_2+a_3x_2^2+a_4x_1+a_5x_2+1 \]
  is torus equivariant under the torus action
  \begin{align*}
    (K^\ast)^2\times K^5&\rightarrow K^5,\\
    ((t_1,t_2),(a_1,\dots,a_5))&\mapsto (t_1^{-2}a_1,t_1^{-1}t_2^{-1}a_2,t_2^{-2}a_3,t_1^{-1}a_4,t_2^{-1}a_5),
  \end{align*}
  since
$t\cdot V(f(a;x)) = V(f(a;t^{-1}\cdot x)) = V(f(t\cdot a;x))$.
  Moreover, $I$ is parametrically independent to the principal ideal generated by

  \[ g(a;x) \coloneqq x_1^2+x_1x_2+x_2^2+x_1+x_2+1, \]
  as $f$ and $g$ share no parameters.
\end{example}

One central result for this paper is the following proposition that expresses the generic root count as a tropical intersection number:
\begin{proposition}\cite[Proposition 4.19]{HelminckRen22}\label{prop:HelminckRen}
  Let $I_1,\dots,I_r\unlhd K[a][x^\pm]$ be para\-metrized polynomial ideals generically of complementary dimension, i.e., $\sum_{i=1}^r\codim(I_{i,K(a)})=n$. Suppose $I_1,\dots,I_{r-1}$ are torus equivariant and $I_1,\dots,I_r$ are parametrically independent. Then the generic root count of their sum $I\coloneqq I_1+\dots+I_r$ is their tropical intersection number for a generic choice of parameters:
  \[ \ell_{I} = \prod_{i=1}^r \Trop(I_{i,P})\quad\text{for }P\in K^m \text{ generic.} \]
\end{proposition}

\section{Tropically transverse systems}\label{sec:tropicallTransverseSystem}
In this section, we introduce the notion of tropically transverse systems in \cref{def:tropicallyTransverseSystem} and use tropical modifications to express their generic root counts in terms of mixed volumes in \cref{thm:tropicallyTransverseSystem}.  We briefly comment on how to check whether a system is tropically transverse and how to compute its generic root count in \OSCAR.

\begin{definition}\label{def:tropicallyTransverseSystem}
	We say a system $F=\{f_1,\dots,f_n\}\subseteq K[a][x^\pm]$ is \emph{horizontally parametrised} if there is a partition $[m]=\coprod_{j=1}^n A_j$ and a finite \emph{support} $S=\{s_1,\dots,s_m\}\subseteq K[x^\pm]$ such that
  \[ f_i = \sum_{j\in A_i} a_j\cdot s_j \qquad\text{for all } i=1,\dots,n. \]

  Furthermore, we say that the support $S$ has a \emph{tropically transverse base}, if there is a $B=\{ b_1,\dots,b_r \}\subseteq K[x^\pm]$ such that
  \begin{enumerate}
  \item for each $j\in [m]$ there is a $\beta_j\in\ZZ^r$ such that $s_j=b^{\beta_j}$.
  \item for any $J\subseteq [r]$ we have
    \[ \codim\Big(\bigcap_{j\in J} |\Trop(b_j)|\Big) = \min\Big(\sum_{j\in J} \codim \big(\Trop(b_j)\big),n\Big), \]
    where $\codim(\bigcap_{j\in J} |\Trop(b_j)|)$ is the codimension of any polyhedral complex supported on $\bigcap_{j\in J} |\Trop(b_j)|$.
  \end{enumerate}
  For the sake of brevity, we will refer to a horizontally parametrised polynomial system with tropically transversal base simply as a \emph{tropically transversal system}.
\end{definition}

\begin{example}
	Consider the parametrized system in \cite[Equation 5.1]{BMMT22}:
	\begin{align*}
		f_i & = a_{1, i} u_i (u_i^2 + v_i^2) + a_{2, i} u_i + a_{3, i} v_i + a_{4, i} + \sum_{j \ne i} c_{j, i} v_j &\text{for }i=1,\dots,N, \\
		g_i & = b_{1, i} v_i (u_i^2 + v_i^2) + b_{2, i} u_i + b_{3, i} v_i + b_{4, i} + \sum_{j \ne i} d_{j, i} u_j &\text{for }i=1,\dots,N,
	\end{align*}
  with parameters $a_{j, i}$, $b_{j, i}$, $c_{j, i}$, $d_{j, i}$, and variables $u_i$, $v_i$.  Its support
  \[ S=\{1,u_i,v_i,u_i(u_i^2+v_i^2), v_i(u_i^2+v_i^2)\mid i=1,\dots, N\}, \]
  has a tropically transverse base $B=\{u_i,v_i,u_i^2+v_i^2\mid i=1,\dots,N\}$, as $\Trop(u_i)=\emptyset=\Trop(v_i)$ and the $\Trop(u_i^2+v_i^2)$ intersect transversely as their polynomials have distinct variables.
\end{example}

\begin{example}
	Consider the parametrised system in \cite[Theorem 5.1]{BBPMZ23}:
	\begin{align*}
		a_0 + a_1 u + a_2 v + a_3 u(u^2 + v^2) + \cdots + a_{n + 1} u {(u^2 + v^2)}^{n} & = 0 \\
		b_0 + b_1 u + b_2 v + b_3 v(u^2 + v^2) + \cdots + b_{n + 1} v {(u^2 + v^2)}^{n} & = 0
	\end{align*}
	with parameters $a_i$, $b_i$ and variables $u$, $v$.
  Its support
  \[ S=\{1,u,v,u(u^2 + v^2),\dots,u(u^2 + v^2)^{n},v(u^2 + v^2),\dots,v(u^2 + v^2)^{n}\}, \]
  has a tropically transverse base $B=\{u,v,u^2+v^2\}$.
\end{example}

\begin{remark}
  It is straightforward to check whether a parametrised system $f_i = \sum_{j\in A_i} a_j\cdot s_j$ as in \cref{def:tropicallyTransverseSystem} is tropically transverse.  We can obtain a minimal base $B=\{b_1,\dots,b_r\}$ by computing the prime factors of all $s_j$.
  Testing whether the $\Trop(b_\ell)$ intersect transversely can be done by computing the regular subdivisions of the Newton polytopes $\Delta(b_1),\dots,\Delta(b_r)$ and $\Delta(\prod_{\ell=1}^r b_\ell)=\sum_{\ell=1}^r\Delta(b_\ell)$ induced by the valuation of their coefficients. Note that every cell $\sigma$ of the subdivision of $\Delta(\prod_{\ell=1}^r b_\ell)$ can be written as a Minkowski sum $\sigma_1+\dots+\sigma_r$ where $\sigma_\ell$ is a cell of the subdivisions of $\Delta(b_\ell)$ \cite[Section 4.2]{Joswig21}.
  The $\Trop(b_\ell)$ intersect transversely if and only if $\sum_{\ell=1}^r \dim(\sigma_\ell) = \dim(\sigma)$ for all cells $\sigma$. \cref{fig:tropicalTransversality} illustrates three examples, two of which do not intersect transversely and with the defective dual cell $\sigma$ highlighted.
\end{remark}

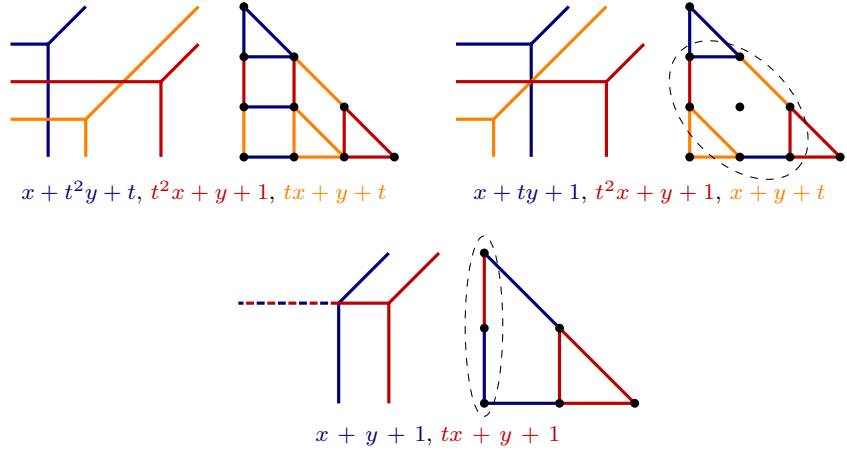
\begin{figure}[h]
  \centering
  \begin{tikzpicture}
    \node[anchor=east,xshift=-3mm] (subfigure1)
    {
      \begin{tikzpicture}
        \node (tropicalPicture)
        {
          \begin{tikzpicture}[scale=0.5]
            \coordinate (blueVertex) at (0,0);
            \coordinate (orangeVertex) at (1,-2);
            \coordinate (redVertex) at (3,-1);
            \draw[blue!50!black,very thick]
            (blueVertex) -- ++(-1,0)
            (blueVertex) -- ++(0,-3)
            (blueVertex) -- ++(1,1);
            \draw[orange,very thick]
            (orangeVertex) -- ++(-2,0)
            (orangeVertex) -- ++(0,-1)
            (orangeVertex) -- ++(3,3);
            \draw[red!75!black,very thick]
            (redVertex) -- ++(-4,0)
            (redVertex) -- ++(0,-2)
            (redVertex) -- ++(1,1);
          \end{tikzpicture}
        };
        \node[anchor=west,xshift=3mm] (dualPicture) at (tropicalPicture.east)
        {
          \begin{tikzpicture}[scale=0.667]
            \draw[blue!50!black,very thick]
            (0,2) -- (1,2) -- (0,3) -- cycle
            (0,0) -- (1,0)
            (0,1) -- (1,1);
            \draw[orange,very thick]
            (1,0) -- (2,0) -- (1,1) -- cycle
            (0,0) -- (0,1)
            (2,1) -- (1,2);
            \draw[red!75!black,very thick]
            (2,0) -- (3,0) -- (2,1) -- cycle
            (0,1) -- (0,2)
            (1,1) -- (1,2);
            \fill
            (0,0) circle (2.5pt)
            (1,0) circle (2.5pt)
            (2,0) circle (2.5pt)
            (3,0) circle (2.5pt)
            (0,1) circle (2.5pt)
            (1,1) circle (2.5pt)
            (2,1) circle (2.5pt)
            (0,2) circle (2.5pt)
            (1,2) circle (2.5pt)
            (0,3) circle (2.5pt);
          \end{tikzpicture}
        };
      \end{tikzpicture}
    };
    \node[anchor=north,yshift=-12mm,text width=50mm,align=center] at (subfigure1)
    { \textcolor{blue!50!black}{$x+t^2y+t$}, \textcolor{red!75!black}{$t^2x+y+1$}, \textcolor{orange}{$tx+y+t$} };
    \node[anchor=west,xshift=0mm] (subfigure2)
    {
      \begin{tikzpicture}
        \node (tropicalPicture)
        {
          \begin{tikzpicture}[scale=0.5]
            \coordinate (blueVertex) at (0,1);
            \coordinate (orangeVertex) at (-1,-1);
            \coordinate (redVertex) at (2,0);
            \draw[blue!50!black,very thick]
            (blueVertex) -- ++(-2,0)
            (blueVertex) -- ++(0,-3)
            (blueVertex) -- ++(1,1);
            \draw[orange,very thick]
            (orangeVertex) -- ++(-1,0)
            (orangeVertex) -- ++(0,-1)
            (orangeVertex) -- ++(3,3);
            \draw[red!75!black,very thick]
            (redVertex) -- ++(-4,0)
            (redVertex) -- ++(0,-2)
            (redVertex) -- ++(1,1);
          \end{tikzpicture}
        };
        \node[anchor=west,xshift=3mm] (dualPicture) at (tropicalPicture.east)
        {
          \begin{tikzpicture}[scale=0.667]
            \draw[blue!50!black,very thick]
            (0,2) -- (1,2) -- (0,3) -- cycle
            (1,0) -- (2,0);
            \draw[orange,very thick]
            (0,0) -- (1,0) -- (0,1) -- cycle
            (2,1) -- (1,2);
            \draw[red!75!black,very thick]
            (2,0) -- (3,0) -- (2,1) -- cycle
            (0,1) -- (0,2);
            \fill
            (0,0) circle (2.5pt)
            (1,0) circle (2.5pt)
            (2,0) circle (2.5pt)
            (3,0) circle (2.5pt)
            (0,1) circle (2.5pt)
            (1,1) circle (2.5pt)
            (2,1) circle (2.5pt)
            (0,2) circle (2.5pt)
            (1,2) circle (2.5pt)
            (0,3) circle (2.5pt);
          \end{tikzpicture}
        };
      \end{tikzpicture}
    };
    \node[anchor=north,yshift=-12mm,text width=50mm,align=center] at (subfigure2)
    { \textcolor{blue!50!black}{$x+ty+1$}, \textcolor{red!75!black}{$t^2x+y+1$}, \textcolor{orange}{$x+y+t$} };
    \node[xshift=40mm,yshift=-3.75mm] (subfigure2Highlight)
    {
      \begin{tikzpicture}
        \draw[dashed,rotate=-45] (0,0) ellipse (1.1cm and 0.7cm);
      \end{tikzpicture}
    };
    \node[anchor=north,yshift=-20mm] (subfigure3)
    {
      \begin{tikzpicture}
        \node (tropicalPicture)
        {
          \begin{tikzpicture}[scale=0.667]
            \coordinate (blueVertex) at (0,0);
            \coordinate (redVertex) at (1,0);
            \draw[blue!50!black,very thick]
            (blueVertex) -- ++(0,-2)
            (blueVertex) -- ++(1,1);
            \draw[red!75!black,very thick]
            (redVertex) -- ++(-1,0)
            (redVertex) -- ++(0,-2)
            (redVertex) -- ++(1,1);
            \draw[very thick,blue!50!black,dash pattern= on 3pt off 5pt,dash phase=4pt]
            (blueVertex) -- ++(-2,0);
            \draw[very thick,red!75!black,dash pattern= on 3pt off 5pt]
            (blueVertex) -- ++(-2,0);
          \end{tikzpicture}
        };
        \node[anchor=west,xshift=3mm] (dualPicture) at (tropicalPicture.east)
        {
          \begin{tikzpicture}
            \draw[blue!50!black,very thick]
            (0,0) -- (1,0)
            (0,0) -- (0,1)
            (0,2) -- (1,1);
            \draw[red!75!black,very thick]
            (1,0) -- (2,0) -- (1,1) -- cycle
            (0,1) -- (0,2);
            \fill
            (0,0) circle (1.667pt)
            (1,0) circle (1.667pt)
            (2,0) circle (1.667pt)
            (0,1) circle (1.667pt)
            (1,1) circle (1.667pt)
            (0,2) circle (1.667pt);
          \end{tikzpicture}
        };
      \end{tikzpicture}
    };
    \node[anchor=north,yshift=-12mm,text width=50mm,align=center] at (subfigure3)
    { \textcolor{blue!50!black}{$x+y+1$}, \textcolor{red!75!black}{$tx+y+1$} };
    \node[xshift=6.125mm,yshift=-32.5mm] (subfigure3Highlight)
    {
      \begin{tikzpicture}
        \draw[dashed] (0,0) ellipse (0.25cm and 1.2cm);
      \end{tikzpicture}
    };
  \end{tikzpicture}\vspace{-3mm}
  \caption{Tropical transversality of hypersurfaces and its dual picture}
  \label{fig:tropicalTransversality}
\end{figure}

\begin{definition}\label{def:tropicallyTransverseModification}
  Let $F=\{f_1,\dots,f_n\}\subseteq K[a][x^\pm]$ be a tropically transverse system as in \cref{def:tropicallyTransverseSystem}.  Its \emph{(tropical) modification} is the parametrised system
  \begin{equation}\label{eq:tropicalTransverseModification}
    \begin{aligned}
      \hat f_i&\coloneqq \sum_{j\in A_i} a_jz_j &&\text{for all }i=1,\dots,n,\\
      \hat g_j&\coloneqq z_j-y^{\beta_j}&&\text{for all }j=1,\dots,m,\\
      \hat h_\ell&\coloneqq y_\ell-b_\ell&&\text{for all }\ell=1,\dots,r.
    \end{aligned}
  \end{equation}
  Here, the $a_j$ remain parameters and the $x_i$, $y_j$, and $z_\ell$ are variables.  We abbreviate the parametrised polynomial ring of the modification by\linebreak $K[a][x^\pm,y^\pm,z^\pm]\coloneqq K[a_1,\dots,a_m][x_1^\pm,\dots,x_n^\pm,y_1^\pm,\dots,y_r^\pm,z_1^\pm,\dots,z_m^\pm]$.  For the sake of consistency all polynomials are regarded as parametrised polynomial equations, even the $g_j$ and $h_\ell$ which do not depend on the parameters. Let $\hat I$ be the ideal generated by the $\hat f_i$, $\hat g_j$ and $\hat h_\ell$.  We denote the ambient space of the tropicalisations by $\RR^{n+r+m}$ with coordinates indexed by the variables.  In particular $e_{x_i}$, $e_{y_\ell}$, $e_{z_j}$ are the unit vectors in $\RR^{n+r+m}$.
\end{definition}

\begin{lemma}\label{lem:transversalIntersection1}
	Let $F$ be a tropically transverse system, and let $\hat h_\ell$ be defined as in~\cref{def:tropicallyTransverseModification}.  Then for all $P\in K^m$
	\[ \Trop\Big(\langle \hat h_{\ell,P} \mid \ell=1,\dots,r \rangle\Big) = \bigwedge_{\ell=1}^r \Trop\big(\hat h_{\ell,P}\big), \]
  where the equality only holds up to refinement of weighted polyhedral complexes.
\end{lemma}
\begin{proof}
	By \cref{thm:transverseIntersection}, it suffices to show that the $\Trop(\hat h_{\ell,P})$ intersect transversely.
	To prove this, consider an intersection point $w\in \bigcap_{\ell=1}^r|\Trop(\hat h_{\ell,P})|$ and let $\sigma_\ell\in\Trop(\hat h_{\ell,P})$ be a maximal polyhedron containing $w$.
	Note that the hypersurfaces intersect transversally around $w$ if and only if their stars around $w$ intersect transverally around $0$.
  We may therefore assume without loss of generality that $w=0$ and that all $\sigma_\ell$ are cones.

	Letting $(\cdot)^\perp$ denote the orthogonal complement of a vector, we then have
	\begin{equation*}
		\spanoperator(\sigma_\ell)=
		\begin{cases}
			(e_{y_\ell}-\sum_{i=1}^n \alpha_i\cdot e_{x_i})^\perp   & \text{for a monomial } x^\alpha \text{ in } b_\ell \text{ (Type Y), or}            \\
			(\sum_{i=1}^n (\beta_i-\alpha_i)\cdot e_{x_i})^\perp & \text{for two monomials } x^\alpha, x^\beta \text{ in } b_\ell \text{ (Type X).}
		\end{cases}
	\end{equation*}
	Arranging all normal vectors above as rows of a matrix, with all $\sigma_\ell$ of Type Y on top and all $\sigma_\ell$ of Type X on the bottom, with all $e_{y_\ell}$ on the left and all $e_{x_i}$ on the right, yields a matrix of the form
	\begin{equation*}
		\left(
		\begin{array}{c|c}
			I & \ast \\ \hline
			0 & B
		\end{array}
		\right)
	\end{equation*}
	where $I$ is a full-rank 0/1-submatrix in row-echelon form, and $B$ is another full-rank submatrix due to $\Trop(b_\ell)$ intersecting transversally.  Hence the matrix is itself of full rank, and the $\sigma_\ell$ intersect each other transversally.
\end{proof}

\begin{lemma}\label{lem:transversalIntersection2}
  Let $F$ be a tropically transverse system, and let $\hat g_j$, $\hat h_\ell$ be defined as in~\cref{def:tropicallyTransverseModification}. Then for all $P\in K^m$
  \begin{align*}
    &\Trop\Big(\langle \hat g_{j,P}, \hat h_{\ell,P} \mid j=1,\dots,m, \ell=1,\dots,r \rangle\Big)\\
    &\hspace{12mm}= \bigwedge_{j=1}^m \Trop\big(\hat g_{j,P}\big) \wedge \bigwedge_{\ell=1}^r \Trop\big(\hat h_{\ell,P}\big),
  \end{align*}
  where the equality only holds up to refinement of weighted polyhedral complexes.
\end{lemma}
\begin{proof}
  We use a similar proof to the one of Lemma \ref{lem:transversalIntersection1}. By \cref{thm:transverseIntersection}, it is enough to show that the hypersurfaces in the above equation intersect transversally. Let $\tau_j\in \Trop(\hat g_{j,P})$, $\sigma_\ell\in \Trop(\hat h_{\ell,P})$ be maximal polyhedra intersecting in a point $w$. Again, we may assume without loss of generality that $w = 0$ and that the $\tau_j$ and $\sigma_\ell$ are polyhedral cones.

  Let $v_j$ and $u_\ell$ be the normal vectors on the linear span of $\tau_j$ and $\sigma_\ell$, respectively. Arranging all normal vectors as a rows of a matrix, with all $v_j$ on top and all $u_\ell$ on the bottom, with all $e_{z_j}$ on the left and all $e_{y_\ell}, e_{x_i}$ on the right, yields a matrix of the form
	\begin{equation*}
		\left(
		\begin{array}{c|c}
			I & \ast \\ \hline
			0 & B
		\end{array}
		\right)
	\end{equation*}
  where the submatrix $I$ is the $m\times m$ identity matrix and the submatrix $B$ is of full rank by~\cref{lem:transversalIntersection1}.  Hence the matrix is of full rank, and the $\tau_j$, $\sigma_\ell$ intersect each other transversely.
\end{proof}

\begin{theorem}\label{thm:tropicallyTransverseSystem}
	Let $F$ be a tropically transverse system, $I$ the parametrised ideal it generates, and let $\hat I$, $\hat f_i$, $\hat g_j$, $\hat h_\ell$ be as defined as in \cref{def:tropicallyTransverseModification}. Then the generic root count equals
  \begin{align*}
    \ell_{I} &= \Big(\prod_{i=1}^n \Trop(\hat f_{i,P})\Big) \cdot \Big(\prod_{j=1}^m \Trop(\hat g_{j,P})\Big) \cdot \Big(\prod_{\ell=1}^r \Trop(\hat h_{\ell,P})\Big)\\
    &= \mathrm{MV}\Big(\Delta(\hat f_{i,P}),\Delta(\hat g_{j,P}),\Delta(\hat h_{\ell,P})\mid i\in [n], j\in [m], \ell\in[r] \Big)
  \end{align*}
  for $P\in K^m$ generic, where $\Delta(\cdot)$ denotes the Newton polytope, and $\mathrm{MV}(\cdot)$ denotes the normalized mixed volume.
\end{theorem}

\begin{proof}
  First note that $I$ and $\hat I$ have the same generic root count.  As $\langle \hat f_i\mid i=1,\dots,n\rangle$ is torus equivariant and parametrically independent to $\langle \hat g_j, \hat h_\ell \mid j=1,\dots,m,\ell=1,\dots,r \rangle$, we have by~\cref{prop:HelminckRen}:
  \begin{align*}
    \ell_{I} &= \Trop(\langle \hat f_{i,P}\mid i=1,\dots,n \rangle) \\
             &\hspace{7mm}\cdot \Trop(\langle \hat g_{j,P}, \hat h_{\ell,P} \mid j=1,\dots,m,\ell=1,\dots,r \rangle).
  \end{align*}
  Next, the $\Trop(\hat g_{j,P})$ and $\Trop(\hat h_{\ell,P})$ intersect each other transversely for all $P\in K^m$ by~\cref{lem:transversalIntersection2}, hence
  \begin{align*}
    \ell_{I} &= \Trop(\langle \hat f_{i,P}\mid i=1,\dots,n \rangle) \cdot \Trop(\langle \hat g_{j,P} \mid j=1,\dots,m\rangle)\\
             &\hspace{7mm}\cdot \Trop(\langle\hat h_{\ell,P} \mid \ell=1,\dots,r \rangle).
  \end{align*}
  The equality with the tropical intersection product then follows from the fact that $\Trop(\hat f_{i,P})$ intersect each other transversally, and the same holds for the $\Trop(\hat g_{j,P})$ and the $\Trop(\hat h_{\ell,P})$. The equality with the normalized mixed volume is \cite[Theorem 4.6.8]{MS2015}.
\end{proof}

\begin{example}\label{ex:code}
  Code for computing the tropical intersection product of the modification in \cref{thm:tropicallyTransverseSystem} can be found in
  \begin{center}
    \url{https://github.com/isaacholt100/generic_root_count}
  \end{center}
  The code currently relies on \texttt{MixedSubdivision.jl} for the computation of the tropical intersection product, it contains the functions
  \begin{itemize}
  \item \texttt{nonlinear\_resonator\_system(n::Int,m::Int)} for constructing the System~\ref{eq:theSystem}
    \item \texttt{generic\_root\_count(F::Vector\{<:MPolyElem\})} for computing the tropical intersection number (or mixed volume) in \cref{thm:tropicallyTransverseSystem}.
  \end{itemize}
   The latter requires a horizontally parametrised system as a vector of polynomials with coefficients in another polynomial ring:

\begin{minted}{jlcon}
julia> include("generic_root_count/src/main.jl")

julia> A,(a0,a1,a2,a3,a4,b0,b1,b2,b3,b4) = QQ["a0","a1","a2","a3","a4","b0","b1","b2","b3","b4"];

julia> R,(u,v) = A["u","v"];

julia> f = a0+a1*u+a2*v+a3*u*(u^2+v^2)+a4*u*(u^2+v^2)^2;

julia> g = b0+b1*u+b2*v+b3*v*(u^2+v^2)+b4*v*(u^2+v^2)^2;

julia> generic_root_count([f,g])
9
\end{minted}
\end{example}

\section{Coupled Nonlinear Oscillators}\label{sec:mainSystem}

This section is dedicated to an alternate proof of the generic number of equilibria of coupled nonlinear oscillators \cite[Theorem 5.1]{BBPMZ23} using elementary tropical geometry. Namely:
\begin{theorem}\label{thm:mainSystem}
  The following parametrised system of polynomial equations has generically $2mn+1$ solutions:
  \begin{equation}
    \label{eq:theSystem}
    \begin{aligned}
      0&= a_0 + a_1x_1 + a_2x_2 + a_3x_1(x_1^m+x_2^m)+\dots+a_{n+2}x_1(x_1^m+x_2^m)^n\\
      0&= b_0 + b_1x_2 + b_2x_2 + b_3x_2(x_1^m+x_2^m)+\dots+b_{n+2}x_2(x_1^m+x_2^m)^n
    \end{aligned}
  \end{equation}
  Here, the $a_i$ and $b_i$ are the parameters, while $x_1$ and $x_2$ are the variables.
\end{theorem}

For that, we consider the following simplified modification of System~\eqref{eq:theSystem}:

\begin{definition}\label{def:modification}
  Throughout the section, let:
  \begin{equation}
    \label{eq:modifiedSystem}
    \begin{aligned}
      f_1&\coloneqq a_0 + a_1x_1 + a_2x_2 + a_3y_1+\dots+a_{n+2}y_n,\\
      f_2&\coloneqq b_0 + b_1x_1 + b_2x_2 + b_3z_1+\dots+b_{n+2}z_n,\\
      g &\coloneqq w - (x_1^m+x_2^m), \\
      p_i&\coloneqq y_i- x_1w^i \qquad\text{for }i=1,\dots,n,\\
      q_i&\coloneqq z_i- x_2w^i \qquad\text{for }i=1,\dots,n,
    \end{aligned}
  \end{equation}
  in the parametrised polynomial ring $K[a_i,b_i\mid i=0,\dots,n+2][x_j^\pm,w^\pm,y_i^\pm,z_i^\pm\mid j=1,2, i=1,\dots,n]$ with parameters $a_i, b_i$ and variables $x_j,w,y_i,z_i$.

  We will denote the ambient space of its fibrewise tropicalisations by $\RR^{2n+3}$ and index its coordinates by the variables above, e.g., we will use $e_{x_j}$, $e_w$, $e_{y_i}$, and $e_{z_i}$ to denote the unit vectors in $\RR^{2n+3}$.
\end{definition}

To prove \cref{thm:mainSystem}, we explicitly compute the tropical intersection number
\[ \underbrace{\Trop(\langle f_1,f_2\rangle)}_{\eqqcolon\Sigma_\lin} \cdot \underbrace{\Trop(\langle g,p_i,q_i\rangle)}_{\eqqcolon\Sigma_\nlin}. \]
This is done in multiple steps:
\begin{description}
\item[\normalfont\cref{lem:SigmaNlin}] Describes the three maximal cells of $\Sigma_\nlin$: $\sigma_0$, $\sigma_1$, and $\sigma_2$.
\item[\normalfont\cref{lem:SigmaLin}] Describes the maximal cells of $\Sigma_\lin$: $\tau_{i,j;k,l}$.
\item[\normalfont\cref{lem:stableIntersectionSigma0}] Shows that $\Sigma_\lin\cap (\sigma_0+u)=\emptyset$ for a fixed perturbation $u$.
\item[\normalfont\cref{lem:stableIntersectionSigma1}] Shows that $\Sigma_\lin\cap (\sigma_1+u)=\emptyset$.
\item[\normalfont\cref{lem:stableIntersectionSigma2}] Describes which three $\tau_{i,j;k,l}\in\Sigma_\lin$ intersect $\sigma_2+u$.
\end{description}
The section ends with the proof of \cref{thm:mainSystem} by computing the intersection multiplicities of the three intersection points from \cref{lem:stableIntersectionSigma2}.

\begin{lemma}\label{lem:SigmaNlin}
  Let $g,p_i,q_i$ be as defined in \cref{def:modification}. Then
  \[ \Sigma_\nlin\coloneqq \Trop\Big(\langle g, p_i, q_i\mid i=1,\dots,n\rangle\Big) \]
  up to refinement consists of the following three maximal cells:
  \begin{align*}
    \sigma_0 &\coloneqq \RR_{\geq 0}\cdot \Big(e_w + \sum_{i = 1}^n i(e_{y_i} + e_{z_i})\Big) + \RR \cdot v_\nlin, \\
    \sigma_1 &\coloneqq \RR_{\geq 0}\cdot \Big(e_{x_2} + \sum_{i = 1}^n e_{z_i}\Big) + \RR \cdot v_\nlin, \\
    \sigma_2 &\coloneqq \RR_{\geq 0}\cdot \Big(e_{x_1} + \sum_{i = 1}^n e_{y_i}\Big) + \RR \cdot v_\nlin\eqqcolon\sigma_\nlin,
  \end{align*}
  where $v_\nlin\coloneqq e_{x_1}+e_{x_2}+m\cdot e_w + \sum_{i=1}^n (mi+1)\cdot (e_{y_i}+e_{z_i})$ is the generator of the lineality space of $\Sigma_\nlin$. The multiplicities of $\sigma_1$ and $\sigma_2$ are $1$.
\end{lemma}
\begin{proof}
  The statement follows straightforwardly from \cref{thm:transverseIntersection}, as the tropical hypersurfaces $\Trop(g)$, $\Trop(p_i)$ and $\Trop(q_i)$ all intersect transversally. Hence, we have up to refinement
  \[ \Sigma_\nlin = \Trop(g)\wedge \bigwedge_{i=1}^n\Big(\Trop(p_i)\wedge\Trop(q_i)\Big). \]
  Note that $\Trop(g)$ consists of three maximal cells of codimension $1$, while the intersection of the other $2n$ hypersurfaces is a linear space of codimension $2n$ in $\RR^{2n+3}$.  Together, they give rise to the three two-dimensional maximal cells $\sigma_0$, $\sigma_1$, and $\sigma_2$.  The multiplicity of $\sigma_1$ and $\sigma_2$ being $1$ follows directly from the definition.
\end{proof}

\begin{definition}\label{def:SigmaLinCells}
  For $\{i,j\}\in\binom{\{-2,\dots,n\}}{2}$, $i<j$, we define the following two polyhedral cones in $\RR^{2n+3}$:
  \begin{align*}
    \theta_{i,j} &\coloneqq \{ \overbrace{(v_{x_1},v_{x_2},v_w,v_{y_1},\dots,v_{y_n},v_{z_1},\dots,v_{z_n})}^{\eqqcolon v_{xwyz}}\in\RR^{2n+3}\mid\\
    &\hspace{12mm} \mathfrak y_i=\mathfrak y_j\leq \mathfrak y_k \text{ for all }k\neq i,j\},\\
    \zeta_{i,j} &\coloneqq \{ (v_{x_1},v_{x_2},v_w,v_{y_1},\dots,v_{y_n},v_{z_1},\dots,v_{z_n})\in\RR^{2n+3}\mid\\
    &\hspace{12mm}  \mathfrak z_i=\mathfrak z_j\leq \mathfrak z_k \text{ for all }k\neq i,j\},
  \end{align*}
  where for $i=-2,\dots,n$
  \begin{equation*}
    \mathfrak y_i\coloneqq
    \begin{cases}
      v_{x_{-i}} &\text{if } i<0,\\
      0 &\text{if } i=0,\\
      v_{y_i} &\text{if } i>0,
    \end{cases}
    \quad\text{and}\quad
    \mathfrak z_i\coloneqq
    \begin{cases}
      v_{x_{-i}} &\text{if } i<0,\\
      0 &\text{if } i=0,\\
      v_{z_i} &\text{if } i>0.
    \end{cases}
  \end{equation*}
  Moreover, for $\{k,l\}\in\binom{\{-2,\dots,n\}}{2}$, $k\in\{i,j\}$, we say $v_{xwyz}\in\RR^{2n+3}$ \emph{satisfies condition} $(k,l)$ of $\theta_{i,j}$ and $\zeta_{i,j}$ to be respectively
  \begin{equation*}
    \begin{cases}
      \mathfrak y_k = \mathfrak y_l &\text{if } \{k,l\}=\{i,j\},\\
      \mathfrak y_k \leq \mathfrak y_l &\text{if } l\notin\{i,j\},\\
    \end{cases}
    \quad\text{and}\quad
    \begin{cases}
      \mathfrak z_k = \mathfrak z_l &\text{if } \{k,l\}=\{i,j\},\\
      \mathfrak z_k \leq \mathfrak z_l &\text{if } l\notin\{i,j\}.\\
    \end{cases}
  \end{equation*}
\end{definition}

\begin{example}
  Below are conditions for selected $\{i,j\}\in\binom{\{-2,\dots,n\}}{2}$:
  \[ \arraycolsep=0pt
    \begin{array}{R{15mm} C{2mm}|C{2mm} L{4mm}C{7mm}L{4mm} C{4mm} L{4mm}C{7mm}L{4mm}}
      (k,l)   &&& \multicolumn{3}{c}{\theta_{-1,0}} && \multicolumn{3}{c}{\zeta_{-1,0}} \\ \hline
      (-1,0)  &&& v_{x_1}&=&0          && v_{x_1}&=& 0 \\
      (0,-1)  &&& 0&=&v_{x_1}          && 0&=& v_{x_1} \\
      (0,-2)  &&&      0&\leq& v_{x_2} && 0     &\leq& v_{x_2}\\
      (-1,-2) &&& v_{x_1}&\leq& v_{x_2} && v_{x_1}&\leq& v_{x_2}\\
      (0,s)   &&&      0&\leq& v_{y_s} && 0     &\leq& v_{z_s}\\
      (-1,s)  &&& v_{x_1}&\leq& v_{y_s} && v_{x_1}&\leq& v_{z_s}\\[8mm]
    \end{array}\quad \begin{array}{R{17mm} C{2mm}|C{2mm} L{4mm}C{7mm}L{4mm} C{5mm} L{4mm}C{7mm}L{4mm}}
      (k,l)          &&& \multicolumn{3}{c}{\theta_{-1,n}} && \multicolumn{3}{c}{\zeta_{-1,n}} \\ \hline
      (-1,n)  &&& v_{x_1}&=&v_{y_n}     && v_{x_1}&=& v_{z_n} \\
      (n,-1)  &&& v_{x_1}&=&v_{y_n}     && v_{x_1}&=& v_{z_n} \\
      (-1,0)  &&& v_{x_1}&\leq& 0      && v_{x_1}&\leq& 0\\
      (n,0)   &&& v_{y_n}&\leq& 0      && v_{z_n}&\leq& 0\\
      (-1,-2) &&& v_{x_1}&\leq& v_{x_2} && v_{x_1}&\leq& v_{x_2}\\
      (n,-2)  &&& v_{y_n}&\leq& v_{x_2} && v_{z_n}&\leq& v_{x_2}\\
      (-1,s)  &&& v_{x_1}&\leq& v_{y_s} && v_{x_1}&\leq& v_{z_s}\\
      (n,s)   &&& v_{y_n}&\leq& v_{y_s} && v_{z_n}&\leq& v_{z_s}
    \end{array}
  \]
  In the left table $0<s\le n$ and in the right table $0<s\le n-1$.
\end{example}

\begin{lemma}\label{lem:SigmaLin}
  Let $f_1,f_2$ be as defined in \cref{def:modification}, and let $\tau_{i,j;k,l}$ be as defined in \cref{def:SigmaLinCells}.
  Then there is a Zariski dense parameter set $U_0\subseteq K^{|a|+|b|}$ such that for all $P\in U_0$ and up to refinement the maximal cells of
  \[ \Sigma_\lin\coloneqq \Trop\Big(\langle f_{1,P},f_{2,P}\rangle\Big) \]
  consists of all $\tau_{i,j;k,l}\coloneqq \theta_{i,j} \cap \zeta_{k,l}$ satisfying the following condition:
  \begin{equation}\label{eqn:thetaAndZetaNonOverlapping}\{i,j\}=\{k,l\}\subseteq \{-2,-1,0\} \quad \text{does not hold.}
  \end{equation}
\end{lemma}
\begin{proof}
  By \cite[Theorem 3.6.1]{MS2015} and torus-equivariance, there is a Zariski dense set $U_0\subseteq K^{|a|+|b|}$ of elements with coordinate-wise valuation $0$ such that for $P\in U_0$
  \[ \Trop\Big(\langle f_{1,P},f_{2,P}\rangle\Big) = \Trop\big(f_{1,P}\big)\stsect \Trop\big(f_{2,P}\big). \]
  Note that the $\theta_{i,j}$ and $\zeta_{k,l}$ from \cref{def:SigmaLinCells} are the maximal cells of $\Trop(f_{1,P})$ and $\Trop(f_{2,P})$, respectively.  Condition~\eqref{eqn:thetaAndZetaNonOverlapping} excludes all $\theta_{i,j}$ and $\zeta_{k,l}$ that do not intersect transversally, the remaining $\tau_{i,j;k,l}$ are maximal cells of the stable intersection by definition.
  Each cell being of multiplicity one follows from the fact that $\Sigma_\lin$ is a tropical linear space.
\end{proof}

\begin{remark}
  Note that the $\tau_{i,j;k,l}$ in \cref{lem:SigmaLin} are not necessarily distinct.  For example, $\tau_{-2,-1;-1,0}=\tau_{-2,0;-1,0}$.  Important for us is only the fact that $\tau_{i,j;k,l}$ is always a maximal cell.
\end{remark}

\begin{lemma}\label{lem:stableIntersectionSigma0}
  Let $\sigma_0$ be defined as in \cref{lem:SigmaNlin} and let $\Sigma_\lin$ be defined as in \cref{lem:SigmaLin}.  Fix a perturbation
  \[ u\coloneqq e_{z_n} - 3n\cdot e_{x_1} - n\cdot e_{x_2}. \]
  Then $\Sigma_\lin\cap (\sigma_0+u)=\emptyset$.
\end{lemma}
\begin{proof}
  Assume there is an intersection point $P\in\Sigma_\lin\cap(\sigma_0+u)$, say $P\in\tau_{i_0,j_0;k_0,l_0}$ for some $\tau_{i_0,j_0;k_0,l_0}\in\Sigma_\lin$.  Since $P\in\sigma_0+u$, there are $c,d\in\RR$, $c\geq 0$, such that
  \begin{align*}
    P &= c \cdot \Big( e_w+ \sum_{i=1}^n i(e_{y_i}+e_{z_i})\Big) \\
      &\hspace{12mm}+ d\cdot\Big(e_{x_1}+e_{x_2}+m\cdot e_w+\sum_{i=1}^n(mi+1)\cdot (e_{y_i}+e_{z_i})\Big)\\
      &\hspace{12mm}+(e_{z_n} - 3n\cdot e_{x_1} - n\cdot e_{x_2}) \\
      &= (d-3n) \cdot e_{x_1} + (d-n)\cdot e_{x_2}+(dm+c)\cdot e_w\\
      &\hspace{12mm} + \sum_{i=1}^n\Big( (ci+d(mi+1))\cdot (e_{y_i}+ e_{z_i})\Big)+e_{z_n}.
  \end{align*}

  First, we show that $-2\notin \{i_0,j_0,k_0,l_0\}$. Assume that $-2\in \{i_0,j_0\}$ or $-2\in\{k_0,l_0\}$, say $-2=i_0$ or $-2=k_0$. But then $d-n>d-3n$ contradicts either Condition $(-2,-1)$ of $\theta_{-2,j_0}$ or Condition $(-2,-1)$ of $\zeta_{-2,l_0}$.

  Second, we prove that $0\notin \{i_0,j_0,k_0,l_0\}$. Assume that $0\in\{i_0,j_0\}$ or $0\in\{k_0,l_0\}$, say $0=i_0$ or $0=k_0$.  Then either Condition $(0,-1)$ of $\theta_{0,j_0}$ or Condition $(0,-1)$ of $\zeta_{0,l_0}$ state that $0\leq d-3n$ and hence $d\geq 3n>0$, and all $y_i$ and $z_i$ coordinates of $P$ are positive. Hence. $\{i_0,j_0\}=\{0,-1\}=\{k_0,l_0\}$, contradicting Condition~\eqref{eqn:thetaAndZetaNonOverlapping}.

  Third, we show that $\{i_0,j_0\}=\{-1,n\}$ and $\{k_0,l_0\}\subseteq \{-1,n-1,n\}$.  We have shown above that both $\{i_0,j_0\}$ and $\{k_0,l_0\}$ contain an index greater than $0$. Without loss of generality, let $j_0,l_0>0$:
  \begin{enumerate}[leftmargin=7mm]
  \item\label{enumitem:dmc} Condition $(j_0,-1)$ of $\theta_{i_0,j_0}$ states $cj_0+d(mj_0+1)\leq d-3n$, which implies $dm+c\leq -\frac{3n}{j_0}<0$.  Condition $(j_0,n)$ of $\theta_{i_0,j_0}$ states $cj_0+d(mj_0+1)\leq cn+d(mn+1)$, which in turn implies $(j_0-n)(dm+c)\leq 0$.  Combining both yields $j_0= n$.

    \noindent Since $i_0\neq j_0$, this also shows that $i_0$ cannot be greater than $0$ (otherwise it is also $n$). Hence $\{i_0,j_0\}=\{-1,n\}$.
  \item As the $y_i$ and $z_i$ coordinates of $P$ are the same for $i=1,\dots,n-1$, we can use the same arguments as above to obtain $l_0=n-1$ in the case $l_0<n$.  Hence $l_0=n-1$ or $l_0=n$, and therefore $\{k_0,l_0\}\subseteq \{-1,n-1,n\}$.
  \end{enumerate}

  \noindent
  Next, note that since $\{i_0,j_0\} = \{-1,n\}$, Condition $(-1,n)$ of $\theta_{-1,n}$ states
  \begin{equation}
    \label{eq:LemmaSigma0}
    d-3n = cn+d(mn+1)\quad\text{or equivalently}\quad -c-dm=3.
  \end{equation}
  Finally, we distinguish between three cases, leading each case to a contradiction:
  \begin{description}[leftmargin=7mm]
  \item[\textcolor{blue!50!black}{$\mathbf{\{k_0,l_0\}=\{-1,n-1\}}$}] Condition $(-1,n-1)$ of $\zeta_{-1,n-1}$ states $d-3n=c(n-1)+d(m(n-1)+1)$.  Combined with $d-3n=cn+d(mn+1)$ from \eqref{eq:LemmaSigma0}, we obtain $dm+c=0$, contradicting $dm+c<0$ established in Step \eqref{enumitem:dmc}.
  \item[\textcolor{blue!50!black}{$\mathbf{\{k_0,l_0\}=\{-1,n\}}$}] Condition $(-1,n)$ of $\zeta_{-1,n}$ states $d-3n=cn+d(mn+1)+1$.  This contradicts \eqref{eq:LemmaSigma0}.
  \item[\textcolor{blue!50!black}{$\mathbf{\{k_0,l_0\}=\{n-1,n\}}$}] Condition $(n-1,n)$ of $\zeta_{n-1,n}$ states $c(n-1)+d(m(n-1)+1)=cn+d(mn+1)+1$ or equivalently $1=-dm-c$. This contradicts \eqref{eq:LemmaSigma0}.\end{description}
\end{proof}

\begin{lemma}\label{lem:stableIntersectionSigma1}
  Let $\sigma_1$ be defined as in \cref{lem:SigmaNlin} and let $\Sigma_\lin$ be defined as in \cref{lem:SigmaLin}.   Fix the same perturbation as in \cref{lem:stableIntersectionSigma0}:
  \[ u\coloneqq e_{z_n} - 3n\cdot e_{x_1} - n\cdot e_{x_2}. \]
  Then $\Sigma_\lin\cap (\sigma_1+u)=\emptyset$.
\end{lemma}
\begin{proof}
  Assume there is an intersection point $P\in\Sigma_\lin\cap (\sigma_1+u)$, say $P\in \tau_{i_0,j_0;k_0,l_0}$ for some $\tau_{i_0,j_0;k_0,l_0}\in\Sigma_\lin$.  Since $P\in\sigma_1+u$, there are $c,d\in\RR$, $c\geq 0$, such that
  \begin{align*}
    P &= c \cdot \Big( e_{x_2} + \sum_{i=1}^n e_{z_i}\Big) + d\cdot\Big(e_{x_1}+e_{x_2}+m\cdot e_w+\sum_{i=1}^n(mi+1)\cdot (e_{y_i}+e_{z_i})\Big)\\
      &\hspace{12mm}+(e_{z_n} - 3n\cdot e_{x_1} - n\cdot e_{x_2}) \\
      &= (d-3n)\cdot e_{x_1} + (c+d-n)\cdot e_{x_2}+dm\cdot e_w \\
      &\hspace{12mm}+ \sum_{i=1}^n\Big( d(mi+1)\cdot e_{y_i}+(d\cdot (mi+1)+c)\cdot e_{z_i}\Big)+e_{z_n}.
  \end{align*}

  First, we show that $-2\notin \{i_0,j_0,k_0,l_0\}$. Assume that $-2\in \{i_0,j_0\}$ or $-2\in\{k_0,l_0\}$, say $-2=i_0$ or $-2=k_0$. But then $c+d-n>d-3n$ contradicts either Condition $(-2,-1)$ of $\theta_{-2,j_0}$ or Condition $(-2,-1)$ of $\zeta_{-2,l_0}$.

  Second, we prove that $0\notin \{i_0,j_0,k_0,l_0\}$. Assume that $0\in\{i_0,j_0\}$ or $0\in\{k_0,l_0\}$, say $0=i_0$ or $0=k_0$.  In both cases, Condition $(0,-1)$ of $\theta_{0,j_0}$ or Condition $(0,-1)$ of $\zeta_{0,l_0}$ state that $0\leq d-3n$ and hence $d\geq 3n>0$.  Then all $y_i$ and $z_i$ coordinates of $P$ are positive. This implies $\{i_0,j_0\}=\{0,-1\}=\{k_0,l_0\}$, which contradicts Condition~\eqref{eqn:thetaAndZetaNonOverlapping}.

  Third, we show that $\{i_0,j_0\}=\{-1,n\}$ and $\{k_0,l_0\}\subseteq \{-1,n-1,n\}$.  We have shown that both $\{i_0,j_0\}$ and $\{k_0,l_0\}$ contain an index greater than $0$. Without loss of generality, let $j_0,l_0>0$:
  \begin{itemize}[leftmargin=7mm]
  \item Condition $(j_0, -1)$ of $\theta_{i_0,j_0}$ states $d(mj_0+1)\leq d-3n$, which implies $d<0$. Condition $(j_0, n)$ of $\theta_{i_0,j_0}$ states $d(mj_0+1)\leq d(mn+1)$, which implies $j_0=n$.
    Since $i_0\neq j_0$, this also shows that $i_0$ cannot be greater than $0$ (otherwise it is also $n$). Hence $\{i_0,j_0\}=\{-1,n\}$.
  \item Suppose $l_0<n$. Condition $(l_0, -1)$ of $\zeta_{k_0,l_0}$ states $d(ml_0+1)+c\leq d-3n$, which implies $d<0$.  Condition $(l_0, n - 1)$ states $d(ml_0+1)+c\leq d(m(n-1)+1)+c$.  Combining both yields $l_0=n-1$.
    Hence $l_0=n-1$ or $l_0=n$, and therefore $\{k_0,l_0\}\subseteq \{-1,n-1,n\}$.
  \end{itemize}

  \noindent
  Next, note that since $\{i_0,j_0\} = \{-1,n\}$, Condition $(-1, n)$ of $\theta_{-1,n}$ states
  \begin{equation}
    \label{eq:LemmaSigma1}
    d-3n = d(mn+1)\quad\text{or equivalently}\quad -dm=3.
  \end{equation}

  \noindent
  Finally, we distinguish between three cases, leading each case to a contradiction:

  \begin{description}[leftmargin=7mm]
  \item[\textcolor{blue!50!black}{$\mathbf{\{k_0,l_0\}=\{-1,n-1\}}$}] Condition $(-1, n - 1)$ of $\zeta_{-1,n-1}$ states $d-3n=d(m(n-1)+1)+c$.  Combined with $d-3n=d(mn+1)$ from \eqref{eq:LemmaSigma1}, we obtain $c=dm$, contradicting $c\geq 0$.
  \item[\textcolor{blue!50!black}{$\mathbf{\{k_0,l_0\}=\{-1,n\}}$}] Condition $(-1, n)$ of $\zeta_{-1,n}$ states $d-3n=d(mn+1)+c+1$.  Combined with $d-3n=d(mn+1)$ from \eqref{eq:LemmaSigma1}, we obtain $c=-1$, again contradicting $c\geq 0$.
  \item[\textcolor{blue!50!black}{$\mathbf{\{k_0,l_0\}=\{n-1,n\}}$}] Condition $(n - 1, n)$ of $\zeta_{n-1,n}$ states $d(m(n-1)+1)+c=d(mn+1)+c+1$, i.e., $-dm=1$. This contradicts $-dm=3$ from \eqref{eq:LemmaSigma1}.\end{description}
\end{proof}

\begin{lemma}\label{lem:stableIntersectionSigma2}
  Let $\Sigma_\nlin$ and $\sigma_\nlin$ be defined as in \cref{lem:SigmaNlin} and let $\Sigma_\lin$ be defined as in \cref{lem:SigmaLin}.  Fix again the following perturbation from \cref{lem:stableIntersectionSigma0}:
  \[ u\coloneqq e_{z_n} - 3n\cdot e_{x_1} - n\cdot e_{x_2}. \]
  Then $|\Sigma_\lin\cap (\sigma_2+u)|= \{P_1,P_2,P_3\}$, where
  \begin{align*}
    P_1 &\coloneqq v_2 - \frac{3}{m}\cdot v_\nlin + u &&\in \relint(\tau_{-1, n; -1, n}),\\
    P_2 &\coloneqq 2n\cdot v_2 - \frac{n+1}{mn}\cdot v_\nlin + u &&\in \relint(\tau_{-2,-1; -1, n}),\\
    P_3 &\coloneqq 2n\cdot v_2 + n\cdot v_\nlin + u &&\in \relint(\tau_{-2, -1; -1, 0})
  \end{align*}
  for $v_2\coloneqq e_{x_1} + \sum_{i = 1}^n e_{y_i}$ and $v_\nlin\coloneqq e_{x_1}+e_{x_2}+m\cdot e_w + \sum_{i=1}^n (mi+1)\cdot (e_{y_i}+e_{z_i})$.
\end{lemma}
\begin{proof}
  It is straightforward to show that $P_1$, $P_2$, and $P_3$ are intersection points and contained in the relative interior of the cones specified above.  It remains to show that $P_1$, $P_2$, and $P_3$ are the only intersection points.

  Consider an intersection point $P\in\Sigma_\lin\cap (\sigma_2+u)$, say $P\in \tau_{i_0,j_0;k_0,l_0}$ for some $\tau_{i_0,j_0;k_0,l_0}\in\Sigma_\lin$.  We will now show that $P=P_1$, $P_2$, or $P_3$.  Since $P\in\sigma_2+u$, there are $c,d\in\RR$, $c\geq 0$, such that
  \begin{align*}
    P &= c \cdot \Big( e_{x_1} + \sum_{i=1}^n e_{y_i}\Big) + d\cdot\Big(e_{x_1}+e_{x_2}+m\cdot e_w+\sum_{i=1}^n(mi+1)\cdot (e_{y_i}+e_{z_i})\Big)\\
      &\hspace{12mm}+(e_{z_n} - 3n\cdot e_{x_1} - n\cdot e_{x_2}) \\
      &= (c+d-3n)\cdot e_{x_1} + (d-n)\cdot e_{x_2}+dm\cdot e_w \\
      &\hspace{12mm}+ \sum_{i=1}^n\Big( (d\cdot (mi+1)+c)\cdot e_{y_i}+d(mi+1)\cdot e_{z_i}\Big)+e_{z_n}.
  \end{align*}

  Suppose $0<j_0$.  Then Condition $(j_0,-2)$ of $\theta_{i_0,j_0}$ states $d(mj_0+1)+c\leq d-n$, which implies $d<0$. And Condition $(j_0,n)$ of $\theta_{i_0,j_0}$ states $d(mj_0+1)+c\leq d(mn+1)+c$, which together imply $j_0=n$.  As the same argument holds true for $i_0$, we get $\{i_0,j_0\}\subseteq\{-2,-1,0,n\}$.

  Suppose $0<l_0<n$. Then Condition $(l_0,-2)$ of $\zeta_{k_0,l_0}$ states $d(ml_0+1)\leq d-n$, which implies $d<0$.  And Condition $(l_0,n-1)$ of $\zeta_{k_0,l_0}$ states $d(ml_0+1)\leq d(m(n-1)+1)$, which together imply $l_0=n-1$.  As the same argument holds true for $k_0$, we get $\{k_0,l_0\}\subseteq\{-2,-1,0,n-1,n\}$.

  Note that we always have $d\leq n$, as otherwise all coordinates of $P$ other than $x_1$ are positive.  This would imply $\{i_0,j_0\}=\{-1,0\}=\{k_0,l_0\}$, contradicting Condition~\eqref{eqn:thetaAndZetaNonOverlapping}.

  We now distinguish between the following cases, showing either that $P=P_1,P_2,P_3$ or deriving a contradiction:
  \begin{description}[leftmargin=5mm]
  \item[\textcolor{blue!50!black}{$\mathbf{(i_0,j_0)=(-1,0)}$}]
    First, Condition $(0,-2)$ of $\theta_{-1,0}$ states $0\leq d-n$, which together with $d\leq n$ implies $d=n$.
    Second, Condition $(0,-1)$ of $\theta_{-1,0}$ states $0=c+d-3n$, hence $c=2n$.  This shows that $P=P_3$.
  \item[\textcolor{blue!50!black}{$\mathbf{(i_0,j_0)=(-2,0)}$}]
    Condition $(0,-2)$ of $\theta_{-2,0}$ states $0=d-n$, which implies that all $y_i$ and $z_i$ coordinates of $P$ are positive.  Then $P\in\zeta_{k_0,l_0}$ and Condition~\eqref{eqn:thetaAndZetaNonOverlapping} imply $\{k_0,l_0\}=\{-2,-1\}$ or $\{k_0,l_0\}=\{-1,-0\}$.  Both cases imply that the $x_1$ and $x_2$ coordinates of $P$ are $0$, i.e., $d-n=c+d-3n=0$.  This implies $c=2n$, showing that $P=P_3$.
  \item[\textcolor{blue!50!black}{$\mathbf{(i_0,j_0)=(0,n)}$}]
    First, Condition $(0,n)$ of $\theta_{0,n}$ states $0=d(mn+1)+c$, which implies that $d\leq 0$.  Second, Condition $(0,-2)$ of $\theta_{0,n}$ states $0\leq d-n$, which in turn implies that $n\leq d\leq 0$, contradicting $n>0$.
  \item[\textcolor{blue!50!black}{$\mathbf{(i_0,j_0)=(-1,n)}$}] Condition $(-1, n)$ of $\theta_{-1, n}$ states $c + d - 3n = d(mn + 1) + c$ which implies $d = -\frac{3}{m}$. We then have
    \[ d(mn + 1) + 1 = -3n-\frac{3}{m} + 1 < -3n-\frac{3}{m} + 3 = d(m(n-1) + 1) < 0.\]
    If $0 \in \{k_0, l_0\}$ (say $0 = k_0$), Condition $(0, n)$ of $\zeta_{0, l_0}$ states $0 \le d(mn + 1) + 1$ which is a contradiction. Similarly, if $n - 1 \in \{k_0, l_0\}$ (say $n - 1 = k_0$) then Condition $(n - 1, n)$ of $\zeta_{n - 1, l_0}$ states $d(m(n - 1) + 1) \le d(mn + 1) + 1$, which is a contradiction. So $0,n-1\notin\{k_0, l_0\}$.

If $-2 \in \{k_0, l_0\}$ (say $-2 = k_0$), Condition $(-2, n)$ of $\zeta_{-2, l_0}$ states $d - n \le d(mn + 1)+1$, which together with $d=-\frac{3}{m}$ leads to a contradiction.

    The only remaining possibility is $\{k_0, l_0\} = \{-1, n\}$. Then Condition $(-1, n)$ of $\zeta_{-1, n}$ states $c + d - 3n = d(mn + 1) + 1$, which together with $c + d - 3n = d(mn + 1) + c$, implies that $c = 1$. Since also $d = -\frac{3}{m}$, this shows $P = P_1$.
  \item[\textcolor{blue!50!black}{$\mathbf{(i_0,j_0)=(-2,n)}$}] Condition $(n, -1)$ of $\theta_{-2, n}$ states that $d(mn + 1) + c \le c + d - 3n$ which implies $d \le -\frac{3}{m}$.  If $0 \in \{k_0, l_0\}$ (say $0 = k_0$), then Condition $(0, n)$ of $\zeta_{0, l_0}$ states $0 \le d(mn + 1) + 1$ which leads to a contradiction to $d\leq -\frac{3}{m}$. If $-2 \in \{k_0, l_0\}$ (say $-2 = k_0$), then Condition $(-2, n)$ of $\zeta_{-2, l_0}$ states $d - n \le d(mn + 1) + 1$ which leads to a contradiction to $d\leq -\frac{3}{m}$. If $n - 1 \in \{k_0, l_0\}$ (say $n - 1 = k_0$), then Condition $(n - 1, n)$ of $\zeta_{n - 1, l_0}$ states $d(m(n - 1) + 1) \le d(mn + 1) + 1$ which leads to a contradiction to $d\leq -\frac{3}{m}$. If $\{k_0, l_0\} = \{-1, n\}$ (say $-1 = k_0$), then Condition $(-1, -2)$ of $\zeta_{-1, l_0}$ and Condition $(-2, -1)$ of $\theta_{-2, n}$ state $c + d - 3n \le d - n$ and $d - n \le c + d - 3n$ hence $d + c - 3n = d - n$, implying $c = 2n$. Condition $(-2, n)$ of $\theta_{-2, n}$ and Condition $(-1, n)$ of $\zeta_{-1, n}$ state that $d - n = d(mn + 1) + c$ and $c + d - 3n = d(mn + 1) + 1$.  Substituting $c=2n$ then leads to a contradiction.
  \item[\textcolor{blue!50!black}{$\mathbf{(i_0,j_0)=(-2,-1)}$}] Condition $(-2, -1)$ of $\theta_{-2, -1}$ states $c + d - 3n = d - n$ which implies $c = 2n$.

    If \textcolor{red!75!black}{$\{k_0, l_0\} = \{-1, 0\}$}, then Condition $(-1, 0)$ of $\zeta_{-1, 0}$ states $c + d - 3n = 0$ which implies $d = n$. If \textcolor{red!75!black}{$\{k_0, l_0\} = \{-2, 0\}$}, then Condition $(-2, 0)$ of $\zeta_{-2, 0}$ states $d - n = 0$ so $d = n$. In both cases, we have $P = P_3$.

    If \textcolor{red!75!black}{$\{k_0, l_0\} = \{-1, n\}$}, then Condition $(-1, n)$ of $\zeta_{-1, n}$ states $c + d - 3n = d(mn + 1) + 1$ which together with $c=2n$ implies $d =-\frac{n + 1}{mn}$. If \textcolor{red!75!black}{$\{k_0, l_0\} = \{-2, n\}$}, then Condition $(-2, n)$ of $\zeta_{-2, n}$ states $d - n = d(mn + 1) + 1$ which together with $c=2n$ implies $d = -\frac{n + 1}{mn}$. In both cases, we have $P = P_2$.

    Five cases for $\{k_0,l_0\}$ remain, all of which we will lead to a contradiction:
    \begin{description}[leftmargin=5mm]
    \item[\textcolor{red!75!black}{$\mathbf{(k_0,l_0)=(-2,n-1)}$ or $\mathbf{(k_0,l_0)=(-1,n-1)}$}] Either by Condition $(-2,n-1)$ of $\zeta_{-2,n-1}$ or by combining Condition $(-1,n-1)$ of $\zeta_{-1,n-1}$ with Condition $(-2,-1)$ of $\theta_{-2,-1}$, we get $d-n = d(m(n-1)+1)$, which implies $-dm=\frac{n}{n-1}$. Independently, Condition $(n-1,n)$ of $\zeta_{k_0,l_0}$ states $d(m(n-1)+1)\leq d(mn+1)+1$, which implies $-dm\leq 1$, contradicting $-dm=\frac{n}{n-1}$.
    \item[\textcolor{red!75!black}{$\mathbf{(k_0,l_0)=(0,n-1)}$}] Condition $(0,n-1)$ of $\zeta_{0,n-1}$ states $0=d(m(n-1)+1)$ which implies $d=0$.  This contradicts Condition $(0,-2)$ of $\zeta_{0,n-1}$ which is $0\leq d-n$.
    \item[\textcolor{red!75!black}{$\mathbf{(k_0,l_0)=(0,n)}$}] Condition $(0,n)$ of $\zeta_{0,n}$ states $0=d(mn+1)+1$, which implies $d<0$.  This contradicts Condition $(0,-2)$ of $\zeta_{0,n}$, which states $0\leq d-n$.
    \item[\textcolor{red!75!black}{$\mathbf{(k_0,l_0)=(n-1,n)}$}] Condition $(n-1,n)$ of $\zeta_{n-1,n}$ states $d(m(n-1)+1)=d(mn+1)+1$, which implies $dm=-1$. This contradicts Condition $(n-1,-2)$ of $\zeta_{n-1,n}$, which states $d(m(n-1)+1)\leq d-n$, implying $dm(n-1)\leq -n$.\end{description}
  \end{description}
\end{proof}

\begin{proof}[Proof of \cref{thm:mainSystem}]
  We compute the generic root count of System~\eqref{eq:theSystem} by computing the tropical intersection number
  \[ \underbrace{\Trop(\langle f_1,f_2\rangle)}_{=\Sigma_\lin} \cdot \underbrace{\Trop(\langle g,p_i,q_i\rangle)}_{=\Sigma_\nlin}, \]
  where $f_1$, $f_2$, $g$, $p_i$, and $q_i$ are from the modified system in \cref{def:modification}.
  \cref{lem:stableIntersectionSigma2} showed that for a suitably chosen pertubation $u\in\RR^{2n+3}$, the only intersecting cells of $\Sigma_{\lin}$ and $\Sigma_{\nlin}+u$ are:
  \[ \tau_{-1,n;-1,n}\cap (\sigma_2+u),\quad \tau_{-2,-1;-1,n}\cap (\sigma_2+u),\quad \tau_{-2,-1;-1,0}\cap (\sigma_2+u). \]
  Moreover, the intersection happens in the relative interior of the cells above, which, by \cref{lem:SigmaNlin} and \cref{lem:SigmaLin}, all have multiplicity $1$.  Hence, it remains to compute the sublattice indices
  \[ [N : N_{\tau_{-1,n;-1,n}}+N_{\sigma_2}],\quad [N : N_{\tau_{-2,-1;-1,n}}+N_{\sigma_2}],\quad [N : N_{\tau_{-2,-1;-1,0}}+N_{\sigma_2}], \]
  where $N\coloneqq \ZZ^{2n+3}$ and $N_\delta\coloneqq N\cap\spanoperator(\delta)$.

  Recall that
  \begin{align*}
    & \spanoperator(\sigma_2) = \spanoperator\Big(e_{x_1}\!+\!\sum_{i=1}^n e_{y_i},e_{x_1}\!+\!e_{x_2}\!+\!m\cdot e_w\!+\!\sum_{i=1}^n(mi\!+\!1)\cdot (e_{y_i}\!+\!e_{z_i})\Big),\\
    &\spanoperator(\tau_{-1,n;-1,n}) = \spanoperator(e_{x_1}\!+\!e_{y_n}\!+\!e_{z_n},e_{x_2},e_w,e_{y_1},\dots,e_{y_{n-1}},e_{z_1},\dots,e_{z_{n-1}}),\\
    &\spanoperator(\tau_{-2,-1;-1,n}) = \spanoperator(e_{x_1}\!+\!e_{x_2}\!+\!e_{z_n},e_w,e_{y_1},\dots,e_{y_n},e_{z_1},\dots,e_{z_{n-1}}),\\
    &\spanoperator(\tau_{-2,-1;-1,0}) = \spanoperator(e_w,e_{y_1},\dots,e_{y_n},e_{z_1},\dots,e_{z_n}).
  \end{align*}

  For the first index, note that
  \begin{align*}
    &N_{\tau_{-1,n;-1,n}}\!+\!N_{\sigma_2}\\
    &\hspace{3mm}=\ZZ\cdot (e_{x_1}+e_{y_n}+e_{z_n})+\ZZ\cdot e_{x_2}+\ZZ\cdot e_w+\sum_{i=1}^{n-1} \ZZ\cdot e_{y_i}+\sum_{i=1}^{n-1} \ZZ\cdot e_{z_i}\\
    &\hspace{8mm}+\ZZ\cdot (e_{x_1}+e_{y_n}) + \ZZ\cdot (e_{x_1}+(mi+1)\cdot e_{y_n}+(mi+1)\cdot e_{z_n}))\\
    &\hspace{3mm}=\ZZ\cdot e_{x_2} + \ZZ\cdot e_w + \sum_{i=1}^{n-1}\ZZ\cdot e_{y_i} + \sum_{i=1}^n\ZZ\cdot e_{z_i}\\
    &\hspace{8mm}+\ZZ\cdot(e_{x_1}+e_{y_n})+\ZZ\cdot(mn\cdot e_{y_n}).
  \end{align*}
  Therefore, $N/(N_{\tau_{-1,n;-1,n}}+N_{\sigma_2})\cong\ZZ^2/\spanoperator_{\ZZ}((1,1),(0,mn))\cong\ZZ/mn\ZZ$ and hence $[N : N_{\tau_{-1,n;-1,n}}+N_{\sigma_2}]=mn$.

  For the second index, note that
  \begin{align*}
    N_{\tau_{-2,-1;-1,n}}\!+\!N_{\sigma_2}
    &=\ZZ\cdot (e_{x_1}+e_{x_2}+e_{z_n})+\ZZ\cdot e_w+\sum_{i=1}^{n} \ZZ\cdot e_{y_i}+\sum_{i=1}^{n-1} \ZZ\cdot e_{z_i}\\
    &\hspace{5mm}+\ZZ\cdot e_{x_1} + \ZZ\cdot (e_{x_1}+(mi+1)\cdot e_{z_n})\\
    &=\ZZ\cdot e_{x_1} + \ZZ\cdot e_w + \sum_{i=1}^n\ZZ\cdot e_{y_i} + \sum_{i=1}^{n-1}\ZZ\cdot e_{z_i}\\
    &\hspace{5mm}+\ZZ\cdot(e_{x_2}+e_{z_n})+\ZZ\cdot(mn\cdot e_{z_n}).
  \end{align*}
  Therefore, $N/(N_{\tau_{-2,1;-1,n}}+N_{\sigma_2})\cong\ZZ^2/\spanoperator_{\ZZ}((1,1),(0,mn))\cong\ZZ/mn\ZZ$ and hence $[N : N_{\tau_{-2,1;-1,n}}+N_{\sigma_2}]=mn$.

  For the third index, note that
  \[ N_{\tau_{-2,-1;-1,0}}\!+\!N_{\sigma_2} =\ZZ\cdot e_w+\sum_{i=1}^n \ZZ\cdot e_{y_i}+\sum_{i=1}^n \ZZ\cdot e_{z_i}+\ZZ\cdot e_{x_1} + \ZZ\cdot (e_{x_1}+e_{x_2}) = N \]
  Therefore, $[N : N_{\tau_{-1,n;-1,n}}+N_{\sigma_2}]=1$.

  Summing all three indices, we obtain a generic root count of $2mn+1$.
\end{proof}

\printbibliography[keyword=specialized-tropical-geometry]
\end{document}